\documentclass{amsart}
\usepackage[utf8]{inputenc}
\usepackage{amsfonts, amssymb, amsmath, amsthm, latexsym}
\usepackage{hyperref}
\hypersetup{
	colorlinks=true,
	linkcolor=blue,
	filecolor=magenta,      
	urlcolor=blue,
}
\usepackage{mathtools}
\usepackage{tikz,tikz-cd}
\usetikzlibrary{positioning,arrows,arrows.meta, decorations.pathmorphing,backgrounds,patterns,decorations.markings,decorations.pathreplacing,decorations.text,shapes.geometric,svg.path}
\usepackage[shortlabels]{enumitem}
\usepackage{caption, subcaption}
\usepackage[capitalise]{cleveref}
\graphicspath{ {img/} }

\usepackage{comment}

\newtheorem{thm}{Theorem}[section]
\newtheorem{lemma}[thm]{Lemma}
\newtheorem{prop}[thm]{Proposition}
\newtheorem{cor}[thm]{Corollary}
\theoremstyle{definition}
\newtheorem{qu}[thm]{Question}
\newtheorem{defn}[thm]{Definition}
\newtheorem{remark}[thm]{Remark}

\newtheorem{example}[thm]{Example}
\newtheorem{construction}[thm]{Construction}

\newtheorem*{thm*}{Theorem}

\newcommand{\mcomment}[1]{}

\newcommand{\blockcomment}[1]{%
 }%

\newcommand{\lam}[1]{\textcolor{olive}{\textbf{LAM:} #1}}
\newcommand{\ajd}[1]{\textcolor{violet}{\textbf{AD:} #1}}
\usepackage[normalem]{ulem}
\newcommand\so{\bgroup\markoverwith	{\textcolor{blue}{\rule[.5ex]{2pt}{0.4pt}}}\ULon} 

\newcommand*{\goodness}{proximacy}%
\newcommand*{\weakgoodness}{weak proximacy}%
\newcommand*{\good}{proximate}
\newcommand{\Flower}{{\sf Flower}}%
\newcommand{\supp}{{\operatorname{{Supp}}}}
\newcommand{\out}{\operatorname{out}}
\newcommand{\lk}{\operatorname{{lk}}}
\newcommand*{\A}{\mathcal{A}}
\newcommand*{\B}{\mathcal{B}}
\newcommand*{\M}{\mathcal{M}} 
\newcommand*{\X}{\mathcal{X}}
\newcommand*{\Y}{\mathcal{Y}}
\newcommand*{\N}{\mathbb{N}}
\newcommand*{\Z}{\mathbb{Z}}
\newcommand{\R}{\mathbb{R}}
\newcommand\emptyword{\varepsilon}
\newcommand{\ve}{\varepsilon }
\DeclareMathOperator{\st}{St} 

\newcommand{\Ccond}{\mathrm{C}}
\newcommand{\Tcond}{\mathrm{T}}
\newcommand{\Dsim}{\mathrm{sim}}
\newcommand{\SR}{\mathrm{sr}}
\newcommand{\spin}{stem}
\newcommand{\botanic}{botanic}
\newcommand{\rc}{\mathfrak{c}}
\newcommand{\red}{\mathrm{red}}
\newcommand{\mupi}{\mu}
\title{Stallings foldings for rational subsets of automatic groups   
}
\normalsize{\thanks{The authors would like to
 thank Marco Linton and Armando Martino for their useful input in the application to surface groups, and Alex Evetts for helpful discussions of convexity notions. The authors acknowledge support by Leverhulme foundation
grant no.RPG-2022-025 during the work on this article.}}
\author[L. Asencio-Mart\'{i}n]{Luc\'{i}a Asencio-Mart\'in}
\address{School of Mathematics, Statistics and Physics, Newcastle University,
Newcastle upon Tyne
NE1 7RU, United Kingdom}
\email{L.Asencio-Martin2@newcastle.ac.uk}

\author[J. Britnell]{John R. Britnell}
\address{School of Mathematics, Statistics and Physics, Newcastle University,
Newcastle upon Tyne
NE1 7RU, United Kingdom}
\email{john.britnell1@newcastle.ac.uk }

\author[A. Duncan]{Andrew Duncan}
\address{School of Mathematics, Statistics and Physics, Newcastle University,
Newcastle upon Tyne
NE1 7RU, United Kingdom}
\email{andrew.duncan@newcastle.ac.uk }

\author[D. Francoeur]{Dominik Francoeur}
\address{Facultad de Ciencias,
Universidad Aut\'onoma de Madrid,
Cantoblanco Ciudad Universitaria,
28049 Madrid, Spain }
\email{dominik.francoeur@uam.es}

\author[S. Rees]{Sarah Rees}
\address{School of Mathematics, Statistics and Physics, Newcastle University,
Newcastle upon Tyne
NE1 7RU, United Kingdom}
\email{sarah.rees@newcastle.ac.uk}
\begin{document}

\begin{abstract}
  Let $G$ be an automatic group with associated regular language $L$. We describe a procedure for constructing an automaton which recognises elements of a given submonoid or rational subset $K$ of $G$. This builds on work of Kharlampovich, Miasnikov and Weil, on the case where $K$ is a subgroup of~$G$.

Our construction succeeds, after sufficiently many iterations, whenever $K$ satisfies a certain convexity property, which we call $L$-\goodness. We show how to test whether the construction is complete in the case that $K$ is a submonoid; we have no such test for the general case of a rational subset $K$.  

We focus particularly on the case of a surface group $G$ of genus $g>1$, where $L$ is the language of geodesic words in the standard generators. We use small cancellation theory to obtain a method for constructing $L$-recognisable submonoids of $G$. 
\end{abstract}\mcomment{JB: added abstract. Moved acknowledgements to end of Section 1.\\ AD: moved acknowlegements to a footnote on the title page.}
\maketitle
\setcounter{tocdepth}{1}
\tableofcontents

\section{Introduction}\label{sec:Introduction} 

This paper   describes a procedure
whose aim is to construct an automaton to recognise  the elements of a specified
submonoid or rational subset of an automatic group. 
\mcomment{AD: broke the sentence into two.}
The techniques involved are similar to those of the classic Stallings' folding algorithm, which applies to subgroups of
free groups,  and its recent generalisations.

More precisely we extend the partial algorithm of \cite[Section 4]{kmw17}, 
which applies to a subgroup $H$ of an automatic group $G$, with associated 
regular language $L$, and which terminates with an automaton recognising $H$ 
precisely when $H$ has the property that it is $L$-\emph{quasi-convex}.
This property, 
introduced in \cite{gs91}, \mcomment{AD: replaced kmw17 by gs91} generalises
the standard notion of quasi-convexity; we state the definition in \cref{subsection: convexity} below, but note here that the notions of quasi-convexity and $L$-quasi-convexity coincide in the case that $G$ is hyperbolic and $L$ is the set of all geodesics.

The procedure that we describe in this paper extends that of \cite{kmw17}
to the setting
of a rational subset $K$ of $G$ that need not be a subgroup, given by a regular language $Q$ over $S$.
In order for the procedure to terminate, we require a condition on $K$ (really a condition on $Q$) that is stronger than $L$-quasi-convexity, which we
call $L$-\emph{\goodness} (\cref{def:GoodQ}). 

If $G$ is a finitely presented group and $K$ an $L$-{\good} rational subset of $G$, then after a finite number of steps, our procedure 
will output an automaton accepting the language of all words in $L$ that map to
$K$; this is \cref{thm:algorithm_terminates}. In the case that $K$ is a submonoid of $G$ and the regular language $Q$ is equal to $T^*$ for some finite set of words $T$, we see (\cref{thm:submonoid_corollary}) that the $L$-\goodness\ condition can be relaxed somewhat, to a condition that we call \emph{weak}-$L$-\emph{\goodness}\ (\cref{def:WeaklyGoodQ}).

In general our procedure does not give an algorithm for finding the set of $L$-representatives of $K$, since we cannot always tell whether it has completed.
However, in the case where the group $G$ is automatic, and $K$ is an $L$-{\good}\ submonoid of $G$ 
(in which case $K$ is necessarily rational), 
we are able to turn the procedure into an algorithm; we describe a test which can be applied after 
each iteration of the procedure, to tell us whether the automaton it has constructed recognises all the $L$-representatives (and so whether the procedure has completed). This result is \cref{thm:subpgalg}.

Finally, we apply our results, together with some small cancellation theory, to submonoids of surface groups, and state a criterion for a submonoid of a surface group to have constructively decidable membership problem (\cref{thm:Dehn_reduced_Q_constructively_decidable}).
We use this criterion to find some examples of
submonoids with decidable membership problem which we believe are not covered in the literature.

The structure of the paper is as follows: in Section 2 we make preliminary definitions and give some background on automata, rational subsets of groups, and convexity.  In Section 3 we introduce the notion of an $L$-{\good}\ subset of a group and establish the basic properties needed
for our purposes. We also define the weak-$L$-\goodness\ property that is sufficient to allow construction of a folding in the case of finitely generated submonoids. 
Section 4 describes our procedure for constructing an automaton which recognises the $L$-representatives of a rational subset of a group $G$, and establishes our general results concerning it. 
In Section 5 we focus on automatic groups, and discuss how to determine when our procedure has completed.
Section 6 contains our treatment of submonoids of surface groups.

In future work we shall address applications of our procedure to arbitrary rational subsets of certain groups (for example, right-angled Artin groups).

\subsection{Background: Stallings folding}
\label{subsec:Introduction}
The concepts of Stallings foldings and Stallings automata, when introduced in 1983 \cite{stallings83}, provided group theorists with a completely new approach to understanding finitely generated subgroups of free groups.
 
Given a free group $G$ with generating set $S$, and a subgroup $H = \langle T \rangle$ of $G$, specified by its finite generating set $T$ of words over $S$,
Stallings' algorithm applies a sequence of `foldings' to a labelled graph 
constructed from $T$, iterating until 
it yields a deterministic finite state automaton (or labelled graph) $\st(H)$,
called the \emph{Stallings automaton} of $H$, 
which recognises the set of all reduced \mcomment{AD: added ``reduced'' and tweaked wording} words over $S$  
that belong to the subgroup $H$.
From an automata theoretic point of view, the folding operations are precisely what is 
needed for determinisation and minimisation of a \mcomment{AD: changes to wording} finite state 
automaton that has the property that the reverse of a  transition on a symbol 
$x$ is a transition  on the symbol $x^{-1}$; following \cite{bs21} we shall 
call such an automaton \emph{involutive}. 
From a graph theoretic point of view, $\st(H)$ is the core of the Schreier graph of $H$.

The Stallings automaton for a subgroup provides more than just a finite object to recognise its elements;
a lot of valuable information can be read from it. Information about the 
subgroup index, conjugates, normaliser, decidability problems regarding 
membership and intersections, separability and much more can be easily 
extracted from it.

\mcomment{AD: altered the last two paragraphs of this section.}
The overall framework of a generalisation of Stallings folding to 
subgroups of groups beyond free groups entails a group $G$ with finite 
generating set $S$ and a finite set $T$ of words over $S$, generating a  
subgroup $H \leq G$. With input  a structure built from $T$,  an operation (folding)  is iteratively performed
until a finite object $\st(H)$,  uniquely associated to $H$, and from which information about $H$ can be extracted,
is reached. 
 The subgroup $H$ needs to 
satisfy certain hypotheses to guarantee that this finite object is reached. For example, in 
the case of free groups $H$ must  be finitely generated.
The work 
of Kharlampovich, Miasnikov and Weil \cite{kmw17}, follows this framework to give  a  generalisation of Stallings foldings for subgroups of
automatic groups, that have the  $L$-quasi-convex property, thereby
solving such problems as subgroup membership for this class of subgroups. 

More recently, Dani and Levicovitz \cite{dl21} used similar techniques to study 
subgroups of right-angled Coxeter groups, using a generalisation of folding 
that yields a Stallings-like object which is a cube complex instead of a graph.
 Information about normality and index is reflected in this cube complex and, in
 the 
case that the subgroup is finitely generated and quasi-convex, the associated cube complex is finite and this information can be explicitly read from that 
finite object. Generalising the results of  \cite{dl21}, Ben-Zvi--Kropholler--Lyman \cite{bkl22} developed a 
Stallings foldings framework to study subgroups of fundamental groups of 
non-positively curved cube complexes. This construction also yields a finite object 
when the subgroup is finitely generated and quasi-convex, and this object too 
encodes information about normality and index, and allows solution of the 
subgroup membership problem.

\section{Definitions and notation}\label{sec:Notation} 

\mcomment{JB: organised this section into formal subsections, rather than just boldface headings}

\subsection{Words, free monoids, free groups: }
\mcomment{\ajd{ made cosmetic changes to this subsection}}Given a set $\Sigma$ we denote by $\Sigma^*$ the set of all words on $\Sigma$,  as well as the free monoid over $\Sigma$ if we want to stress the concatenation operation.  A set $\Sigma$ equipped with a map $\nu$ to a monoid $M$ is called  a \emph{monoid generating set} for $M$ if the natural extension of $\nu$ to a map from $\Sigma^*$  to $M$ is surjective. Usually all mention of $\nu$ is suppressed and often we implicitly assume that $\Sigma$ is a subset of $M$ and $\nu$ is the inclusion map. The definition of a \emph{group generating set} for a group $G$ is made the same way, using  the free group on $\Sigma$ instead of the free monoid on $\Sigma$.  The symbols $\emptyword$ and $1$ are respectively used to refer to the empty word and to the identity element (in a monoid or a group).

Throughout this document, unless stated otherwise, $S = \{s_1,\ldots s_r\}$ will denote a finite set (also referred to as an \emph{alphabet} with elements
called \emph{letters}).
\mcomment{AD: added formal inverse of a word and set} The \emph{formal inverse} of a word $w=s_1\cdots s_n$, where $s_i\in S^\pm$, is the word
$w^{-1}=s_n^{-1}\cdots s_1^{-1}$. For a subset $T$ of $(S^{\pm})^*$ we define the \emph{formal inverse} of $T$ to be $T^{-1}=\{w^{-1}\,:\, w\in T\}$.
 We denote by $S^\pm = S\sqcup S^{-1}$ the alphabet given by the letters in $S$ together with their (formal) inverses. 
and by $(S^\pm)^*$ the set of all worda over $S^\pm$. 
 The free group on $S$, denoted $F(S)$, 
can be described on the set of equivalence classes of $(S^\pm)^*$ under the 
relation induced by free reduction,
with concatenation as the group operation. In this way, $F(S)$ can be identified with the set of freely reduced words in $(S^\pm)^*$, and hence 
\mcomment{SR: reworded description of $F(S)$}
 may be regarded as a subset of  the free monoid over $S^\pm$. However we note that 
this is an embedding of $F(S)$ in $(S^\pm)^*$ only as a set, and not as a submonoid.

The freely reduced representative of a word $w$ in $(S^\pm)^*$ is denoted by  $\overline w$. 
For a word $w$ in $(S^\pm)^*$, we denote by $|w|$ its length,
given by the number of letters in $w$,  and by $|w|_{F(S)}$  the length $|\overline w|$ of  $\overline w$.
 
For a group $G$ with generating set $S$, we denote by $\mu\colon (S^\pm)^*\to G$ and $\pi\colon F(S)\to G$ the natural projections 
from the free monoid and the free group, respectively,  onto $G$. Note that $\mu$ factors through $\pi$.
\begin{center}
	\begin{tikzcd}
		(S^\pm)^*\arrow[r,   "\overline{ (-)}"]\arrow[dr, pos =3/8, "\mu"]&F(S)\arrow[d,  "\pi"]\\
		&  G = \langle S|R\rangle
		
	\end{tikzcd}
\end{center}
Understanding $F(S)$ as a subset of $(S^\pm)^*$ we can think of $\pi$ as the restriction $\mu\vert_{F(S)}$ of $\mu$ to $F(S)$, and therefore of $\pi^{-1}(G)$ as a subset of  $(S^\pm)^*$. To signify that two words $w_1$, $w_2$ in $(S^\pm)^*$ or $F(S)$ have the same image under $\mu$ or $\pi$, we will write $w_1=_Gw_2$.

Given a set of words $T\subset (S^\pm)^*$, we denote by $T^\pm$ the union of $T$ with the set $T^{-1}$ of its formal inverses; each element of the 
subgroup $\langle T \rangle$ can be represented as a product of elements of $T^\pm$. 
For $w$ in $(T^\pm)^*$ (representing an element of the subgroup $\langle T\rangle$)
we define the length of $w$ in $G$ with 
respect to $T$ to be  $$|w|_{G, T} = \min\{n \colon w =_G v, \: v = h_1\ldots h_n, \: h_i\in T^\pm\}.$$

\subsection{Finite state automata: }
\mcomment{AD: reorganised this subsection to go: automata; necessary results and genralisations; automata with groups and monoids}
We refer to \cite{hrr17} for notation and standard results, noting that more detail can be found in the classic text \cite{hu79}. The notation $\A= (\Sigma, V, E, V_F, v_0)$ 
will be used to describe a \emph{finite state automaton} (FSA) where $\Sigma$ is a finite set (the \emph{alphabet}),  $V$ is the set of vertices or 
\emph{states}, 
$E$ is the set of $\Sigma$-labelled edges or \emph{transitions} of the form $(v, s, w)\in V\times \Sigma\times V$, 
$v_0\in V$ is the \emph{initial state} and $V_F\subset V$ is the set of 
\emph{accepting states} of the automaton. 
As our notation suggests, we shall often view $\A$ as an edge-labelled directed graph. We say that $(v, s, w)$ is 
an \emph{$s$-labelled} transition, with $v$ as its \emph{source} state and $w$ as its \emph{target} state. For an edge $e = (v, s, w)$, we will use the notation $l(e)=s$ for the labelling function. Unless stated otherwise, the word automaton will implicitly refer to one with finitely
many states, throughout this text.

To each finite state automaton $\A$ we can assign the set $L(\A)\subset \Sigma^*$ 
of words \emph{accepted} by the automaton, which are the words given by the concatenation of edge labels of paths in $\A$ that start in $v_0$ and end in one of the 
accepting states;  we call $L(\A)$ the \emph{language} accepted by $\A$,
and say that $\A$ \emph{recognises} $L(\A)$.
 We say that a  subset $L$ of $\Sigma^*$ is \emph{regular} if it is the language $L(\A_L)$  accepted by a finite state automaton $\A_L$.

We can consider automata with $\varepsilon$-transitions, that is, 
automata over the alphabet $\Sigma_{\varepsilon}= \Sigma\sqcup \{\varepsilon\}$. 
We note that,
given any automaton $\A_{\varepsilon}$ over the alphabet $\Sigma_{\varepsilon}$,
an algorithm  described in  
\cite[Section 2.5.1]{hrr17} 
constructs an automaton $\A$ over $\Sigma$ accepting the same language.
	
Similarly, given a \emph{non-deterministic FSA} (one where $s$-labelled transitions $(v, s, w)$, $(v, s, w')$ with different target states $w\neq w'$ are allowed) 
an algorithm described in \cite[Proposition 2.5.2]{hrr17}
constructs
a \emph{deterministic FSA} 

From the above it is clear that there are many different FSA accepting a given
 regular set. However, there is a unique one that is deterministic and has 
the least number of states, which we call the \emph{minimal automaton}, and there is 
an algorithm to obtain it from a non-minimal one (Myhill--Nerode) \cite[Theorem 2.5.4]{hrr17}.
Note that the minimal automaton of a non-empty language is \emph{trim}, 
that is each state lies on a directed path
from the start state to an accept state. \mcomment{High level committee: changed definition of trim... at least twice.} 

Where $M$ is a monoid with finite generating set $\Sigma$, we call a subset $K$ of $M$ \emph{rational} if it is the image of a regular subset $Q$ of $\Sigma^*$ under
\mcomment{\ajd{the canonical}} the canonical monoid homomorphism $\mu: \Sigma^* \to M$. We note that this is one of a few equivalent
definitions of a rational subset of a monoid, and that a subset of $\Sigma^*$ is regular precisely when it is rational.

Since we want to use automata to
define
languages that are subsets of groups, the alphabets $\Sigma$ of our automata will often be sets of the form $S^\pm= S \sqcup S^{-1}$, where $S$ is a finite generating set of a related group.
\mcomment{SR:  deleted sentence suggesting we should abuse notation and accept $S$-labelled might mean $S^\pm$-labelled}

It is well known that if a subset $L$ of $(S^\pm)^*$ is regular, its \emph{free reduction}
$\overline L := \{\overline w\,:\, w\in L\}$  is also regular (Benois' theorem, see \cite{benois,bs21},).\mcomment{AD: added defn of $\overline L$}

\subsection{Rational structures and automatic groups: } 
A pair $(G, L)$ is called a \emph{rational structure} for a group $G$ 
if $G=\langle S\rangle$ for a finite alphabet $S$ and $L\subset (S^\pm)^*$ 
is a regular set where the restriction  $\mu\vert_{L}\colon L\to G$ is a surjection. 
We note that if $(G,L)$ is a rational structure for $G$, then  $(G,\bar{L}$) is also a rational structure, with $\bar{L} \subset F(S)$. 

Given a rational structure $(G, L)$ and a subset $K$ of $G$, we define the  
set of \emph{$L$-representatives} of $K$ to be the subset 
$L\cap \mu^{-1}(K)$ of $(S^\pm)^*$
and the 
set of \emph{reduced $L$-representatives} of $K$ to be the subset 
$\bar{L}\cap \pi^{-1}(K)$ of $F(S)$; 
we observe that the image in $G$ of either of these sets is equal to $K$, 
since $L$ maps onto $G$.
We say that the subset  $K$ is \emph{ $L$-recognisable}\footnote{$L$-recognisable subsets first appear in \cite{gs91}, where they are
  called $L$-rational.} \mcomment{AD: added footnote}
if its $L$-representatives form a regular subset of $(S^\pm)^*$ 
(equivalently, its reduced $L$-representatives form a regular subset of $ F(S)$); 
we abbreviate 
$(S^\pm)^*$-recognisable as recognisable. \cref{fig:rational_subset} illustrates the set of $L$-representatives of a rational subset
$K=\mu(Q)$ of a group, where $Q$ is a regular subset of $(S^\pm)^*$.

\begin{figure}[ht]
\def\firstellip{(0, 0) ellipse [x radius=2, y radius=2.5, rotate=0]}
\def\secondellip{(-.5, 0) ellipse [x radius=1, y radius=1.5,rotate=0]}
\def\thirdellip{(.8, 0) ellipse [x radius=1, y radius=1.2,rotate=0]}
\def\fourthellip{(.8, 0) ellipse [x radius=.5, y radius=.5, rotate=0]}
\def\bounding{(-3,-3) rectangle (3,3)}

	\centering
\begin{tikzpicture}[scale=1]
  \filldraw[fill=none, color=black, opacity=0] \bounding;


\draw \firstellip node[below] {};
\draw \secondellip node[below] {};
\draw \thirdellip node[below] {};
\draw \fourthellip node[below] {};

\draw (0,2) node[draw=none,fill=none, color=black]  {\scriptsize{$(S^{\pm})^*$}};
\draw (-.6,1) node[draw=none,fill=none, color=black]  {\scriptsize{$L$}};
\draw (1,.7) node[draw=none,fill=none, color=black]  {\scriptsize{$\mu^{-1}(K)$}};
\draw (1.1,.1) node[draw=none,fill=none, color=black]  {\scriptsize{$Q$}};

\end{tikzpicture} 

\caption{$L$-representatives of a rational subset: $L$ and $Q$ are regular subsets of $(S^\pm)^*$ and the projection $\mu\colon L\to G$ makes $(G, L)$ a rational structure. The subset $K=\mu(Q)$ is rational in $G$, and its set of $L$-representatives is $\mu^{-1}(K)\cap L$ which is not, in general, either
  regular or equal to $Q\cap L$.}  
	\label{fig:rational_subset}
\end{figure}
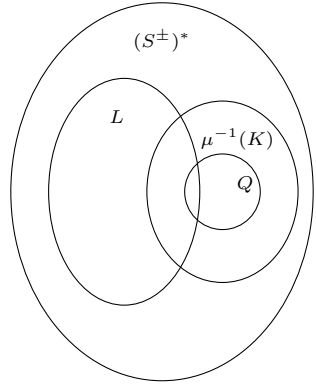

Following \cite{echlpt92}, we call a group $G$ \emph{automatic} if a rational structure $(G, L)$ exists, which is additionally equipped with a collection of 
finite state automata $\{\M_s\}$ for  $s\in S^\pm\cup \{1_G\}$ called 
\emph{multipliers}, that accept strings corresponding to pairs of words $(w_1, w_2)$ in $L\times L$ such that $w_1s =_G w_2$ in the group. 

A subset $K$ of a finitely presented group $G$ 
has \emph{decidable membership problem} if there exists an algorithm for deciding whether or not an element $g\in G$ lies in $K$.
The following proposition derives from \cite[Proposition 3.13]{kmw17}.
\begin{prop}\label{prop:AutomaticLRecognisableMembership}
Let $G=\langle S\rangle$ be a group with automatic structure $(G,L)$, and let $K$ be a rational subset of $G$. If $K$ is $L$-recognisable, then $K$ has decidable membership problem.
\end{prop}
\begin{proof}
Let $h\in G$. Lemma 3.10 of \cite{kmw17} implies that it is possible algorithmically to construct an automaton $B_h$ which accepts precisely the $L$-representatives of $h$. If $K$ is $L$-recognisable, then there exists an automaton $A$ which accepts the $L$-representatives of $K$. Now $h\in K$ if and only if the languages of $A$ and $B_h$ have non-empty intersection.
\end{proof}
We note that while \cref{prop:AutomaticLRecognisableMembership} establishes that an algorithm exists to decide membership of $K$, it does not enable us to run that algorithm unless we can first calculate the automaton $A$ whose language is the set of $L$-representatives of $K$. With this in mind we say that
a decision problem is \emph{constructively decidable} if it is decidable and we have an effective means to construct an algorithm to solve it. \mcomment{AD: added constructively decidable}

\subsection{Conventions for Cayley graphs and paths: } \label{sec:cayley} Given a group $G =\langle S\rangle$, 
the \emph{Cayley graph} of $G$ with respect to $S$ will be denoted by $\Gamma(G, S)$ (or just  $\Gamma(G)$ if the generating set $S$ is understood).
The graph is directed and labelled.
Its vertex set $V\Gamma$ is in correspondence with $G$ and the label $l(v)$ 
of a vertex $v$ is the corresponding element of $G$.
For each $v \in V\Gamma$, and each $s \in S^\pm$, there is a directed edge 
$e\in E\Gamma \subseteq V\Gamma \times V\Gamma$ with label $s$, joining $v$ to the vertex with label $l(v)s$.

By a path (of \emph{length} $n\ge 0$)\mcomment{AD: added length and lower bound for $n$} in the Cayley graph, we mean a directed graph homomorphism $\gamma\colon P_n\to \Gamma(G, S)$ from the path graph $P_n$ into the Cayley graph.  
We often think of $\gamma$ as its image in $\Gamma(G, S)$, rather than as a homomorphism. \label{page:path}

For each edge $e = (u, v)$ in the Cayley graph, labelled by a generator $s$, the graph has an inverse edge $e^{-1}=(v, u)$ with label $s^{-1}$; 
similarly, given a path $\gamma$ in the graph from a vertex $u$ to a vertex $v$ we 
can assign to it a label $l(\gamma)$ in  $(S^\pm)^* $ given by the concatenation 
of the labels of edges in the path. Each such $\gamma$ has a natural inverse path $\gamma^{-1}$ whose label $l(\gamma^{-1})$ is the formal inverse $l(\gamma)^{-1}$ of $l(\gamma)$. 

Since $V\Gamma$ is in bijection with $G$ through the labelling function, we  will often abuse notation and simply refer to a vertex $v$ by its label 
$g=l(v)$.
 If the initial vertex for a path is fixed (e.g. the path starts at $1_G$) 
the set of words in $(S^\pm)^*$  is in bijection with the set of (not necessarily reduced) paths in the Cayley graph, so we will often interchangeably talk about a path $\gamma$ or its label $w=l(\gamma)$. 

Note that the length $n$ of the path $\gamma\colon P_n\to \Gamma(G, S)$ with label $l(\gamma)=w$ coincides with length $|w|$ of the word
$w$ in the free monoid.\mcomment{AD: reworded}
Given an index $0\leq i\leq |w|$, we denote by $w_i$ the $i$-th vertex of the 
path and by $w(i)$ the prefix of $w$ containing its first $i$ letters, that is, $w(i)=\gamma\vert_{P_i}$. When $\gamma$ is a path starting in $1_G$, this means that $w_i = \mu(w(i))$.

Lastly, taking the usual graph metric where each edge has length 1 we can think of the Cayley graph of a group $G$, and therefore of $G$ itself, as a metric space whose metric will be denoted by $d\colon G\times G\to \mathbb N$ throughout this text. 

\subsection{Convexity notions: } \label{subsection: convexity}
Taking into account the bijection between a group $G = \langle S\rangle$ and the vertices of the graph $\Gamma(G, S)$, it is natural to define a subset $K\subset G$ to be \emph{convex} when, for any pair of elements $g, h\in K$, the vertices of any geodesic in $\Gamma(G, S)$ between $g$ and $h$ are all contained in $K$. 

In practice, it is useful to work with a weaker condition than convexity:\mcomment{JB:  replaced complexity with convexity}
we define a subset  $K$ of $G$ to be \emph{quasi-convex} with constant $k$ if 
each geodesic in $\Gamma(G, S)$ connecting elements $g$ and $h$ of $K$ is contained 
in a $k$-neighbourhood of $K$ \cite[III.$\Gamma$.3]{bridson11}; 
we say that the subset is quasi-convex if such a constant $k$ exists. 
In particular, the study of quasi-convex subgroups has been fruitful in 
the last decades (see \cite{bridson11, dl21, bkl22}).

In the case where the group $G$ has a rational structure $L\subset (S^\pm)^*$, 
a more general concept is defined to work with subgroups of $G$: a subset $K$ 
is said to be \emph{$L$-quasi-convex} with constant $k$ if, for any $g\in K$,
 every  path in $\Gamma(G, S)$ from $1_G$ to $g$ defined by a word in $L$ 
(in other words, every $L$-representative of  $K$) lies in the $k$-neighbourhood of $K$ \cite{gs91, kmw17}. 
Note that this definition is particularly powerful 
when $K$ is a subgroup, 
as closure under inversion allows us to translate any path between elements $g$ and $h$ to a path between $1_G$ and the element $g^{-1}h$ in the subgroup. This is not the case for other subsets, e.g. rational subsets. In the case of subgroups, $L$-quasi-convexity is equivalent to $L$-recognisability \cite[Theorem 2.2]{gs91}.

Lastly, to point out some relations between quasi-convexity and $L$-quasi-convexity, let $L_{\texttt{geod}}$ be the language of geodesics in $\Gamma(G,S)$. 
If $L = L_{\texttt{geod}}$, then a subgroup is quasi-convex if and only if it is
 $L$-quasi-convex; if $L_{\texttt{geod}}\subseteq L$, then an $L$-quasi-convex 
subgroup must be quasi-convex; if $L \subseteq L_{\texttt{geod}}$ and $K$ is
a quasi-convex subset containing $1_G$, then $K$ must be  $L$-quasi-convex.
				
As we noted, the definition of $L$-quasi-convexity is very useful when working with subgroups but a generalisation is needed to work with the rational subsets
considered in this article: this motivates our definition of
$L$-{\goodness} in \cref{sec:Setup}.

\section{$L$-\goodness}\label{sec:Setup} 
Let $\left(G, L\right)$ be a rational structure for a finitely presented group $G = \langle S|R\rangle$  and let $\mu \colon (S^\pm)^* \to G$
be the projection from the free monoid  to $G$. 
Given $Q \subseteq (S^\pm)^*$ with $K=\mu(Q)$ a rational subset of $G$, our goal is to build an automaton that accepts the set
of $L$-representatives of $K$ .

This will not always be possible (note that subgroups are rational subsets, and the subgroup membership problem is not always decidable), but if $K$ is $L$-{\good}\  (as defined in \cref{def:GoodQ}) 
then there exists a procedure 
 that eventually produces such an automaton, 
as we will see in \cref{sec:Construction}. If the rational subset is a 
submonoid given by a finite generating set, then a weaker hypothesis of {\weakgoodness} 
will be sufficient ( as defined in \cref{def:WeaklyGoodQ}) 

\mcomment{AD: reworded}
The definition of $L$-{\goodness} involves not only a subset $K$ of $G$ 
but a regular subset of $(S^\pm)^*$ mapped by $\mu$ onto $K$. 
\mcomment{SR: rewrote defn, left in the terms `fellow travel' etc., not sure if we wanted to do that. AD: I think we did.}
 
\mcomment{SR: I don't know what the reference to Fig 2 means, right at the beginning of this definition; that needs to be explained}
\begin{defn}\label{def:GoodQ}
	 Let $(G, L)$ be a  rational structure for a group $G = \langle S \rangle$, 
and let $Q$ be a regular subset of $(S^\pm)^*$. \mcomment{AD: added $Q$  regular} 
We say that $Q$  is $L$-\emph{\good}
over \mcomment{AD: ``in $G$'' changed to ``over $G$''}
$G$ with constants $k$, $c$ if 
for every word $w$ in $(S^\pm)^*$ that
is an $L$-representative of an element in $\mu(Q)$, there exists a word $w'\in Q$ with $\mu(w) = \mu(w')$ such that 
the paths in the Cayley graph of $G$ based at $1_G$ and labelled $w$ and $w'$,  
 \emph{$k$-fellow travel up to at most $c$-reparametrisation}. 
 That is to say, there exists a function $f\colon [0, |w|] \to [0, |w'|]$ satisfying $f(0)=0$, $f(|w|)=|w'|$ and $0\leq f(i+1)-f(i)\leq c$, and such that, using the distance $d$ in the Cayley graph $\Gamma(G, S)$, $d(w_i, w'_{f(i)}) \leq k$ for all $0\leq i \leq |w'|$. (See Figure \cref{fig:our-L-qc}.)
	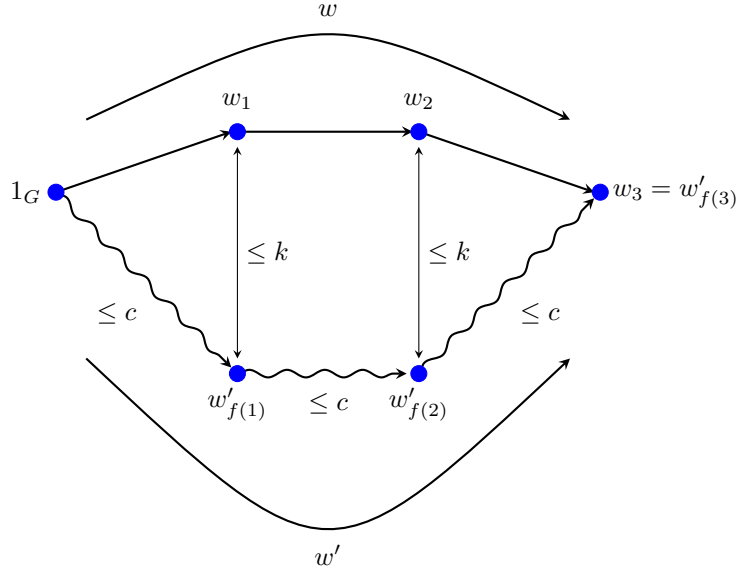
\begin{figure}[ht]
		\centering
\begin{tikzpicture}[thick,scale=.8,arrowmark/.style 2 args={decoration={markings,mark=at position #1 with \arrow{#2}}}]%
  \tikzstyle{every node}=[circle, draw, fill=blue, color=blue,
                        inner sep=0pt, minimum width=6pt]
 \begin{pgfonlayer}{background}                       
                        \draw[->,>=stealth,thick] (0,0) -- (2.9,1);
                        \draw[->,>=stealth,thick] (3,1)--(5.9,1);
                        \draw[->,>=stealth,thick] (6,1)--(8.9,0);
                        \draw[->,>=stealth,thick,decorate,decoration={snake, segment length=5mm, amplitude=.5mm}] (0,0)--(2.9,-2.9);
                        \draw[->,>=stealth,thick,decorate,decoration={snake, segment length=5mm, amplitude=.5mm}] (3,-3)--(5.8,-3);                                            \draw[->,>=stealth,thick,decorate,decoration={snake, segment length=5mm, amplitude=.5mm}] (6,-3)--(8.9,-.1);                      
\end{pgfonlayer}                      
\draw (0,0) ++(-.5,0) node[draw=none,fill=none, color=black] {$1_G$};
\draw (9,0) ++(+1.25,0) node[draw=none,fill=none, color=black] {$w_3=w'_{f(3)}$};
\foreach \x in {0,3}{
  \draw (\x*3,0) node (\x) {};
 }
 \foreach \x in {1,2}{
   \draw (\x*3,1) node (\x) {};
   \draw (\x*3,-3) node (-\x) {};
   \draw (\x*3,0) ++(0,1.5) node[draw=none,fill=none, color=black] {$w_{\x}$};
   \draw (\x*3,-3) ++(0,-.5) node[draw=none,fill=none, color=black] {$w'_{f(\x)}$};
 }
 \draw[<->,>=stealth] (3,.75) -- (3,-2.75);
 \draw[<->,>=stealth] (6,.75) -- (6,-2.75);
 \draw (1,-2) node[draw=none,fill=none, color=black]  {$\le c$};
 \draw (8,-2) node[draw=none,fill=none, color=black]  {$\le c$};
 \draw (4.5,-3.5) node[draw=none,fill=none, color=black]  {$\le c$};
 \draw (3.5,-1) node[draw=none,fill=none, color=black]  {$\le k$};
 \draw (6.5,-1) node[draw=none,fill=none, color=black]  {$\le k$};
 \draw[thick,->,>=stealth,decorate,decoration={bent,amplitude=15mm}] (.5,1.2) -- (8.5,1.2);
 \draw (4.5,3) node[draw=none,fill=none, color=black]  {$w$};
 \draw[thick,->,>=stealth,decorate,decoration={bent,amplitude=-30mm}] (.5,-2.75) -- (8.5,-2.75);
 \draw (4.5,-6) node[draw=none,fill=none, color=black]  {$w'$};
\end{tikzpicture}
          \caption{Illustration of the paths in the Cayley graph defined in \cref{def:GoodQ}.}
		\label{fig:our-L-qc}
	\end{figure}
\end{defn}


In particular, $w$ and $w'$ \emph{asynchronously fellow travel} 
(see \cite{echlpt92} for a definition),
with a (linear) restriction on their asynchronicity. A weaker notion of convexity, useful when working with submonoids of groups, is obtained by dropping the constant $c$:

\begin{defn}\label{def:WeaklyGoodQ}
  Let \mcomment{JB: specified that the function $f$ is non-decreasing. Changed a $\bar{w}$ to $w$.\\[1em]
  AD: added $Q$ regular}
 $(G, L)$ be a  rational structure for a group 
$G = \langle S \rangle$, and let 
 $ Q$ be a regular  subset of  $(S^\pm)^*$.   We say  that $Q$  is 
\emph{weakly}-$L$-\emph{\good} over $G$ with constant $k$ if 
for every word 
$w$ in $L\cap\mu^{-1}(\mu(Q))$ there is a word $w'\in Q$  
and a non-decreasing map $f\colon [0, |w|] \to [0, |w'|]$ satisfying $\mu(w) = \mu(w')$, $f(0)=0$, $f(|w|)=|w'|$  and, using the distance $d$ in the Cayley graph $\Gamma(G, S)$, $d(w_i, w'_{f(i)}) \leq k$ for all $0\leq i \leq |w'|$. 
In other words,
the paths in the Cayley graph of $G$ represented by $w'$ and $w$ must $k$-fellow travel asynchronously. 

\end{defn}

\begin{remark}
Note that we may have distinct subsets of $(S^\pm)^*$ that map onto the same 
subset of $G$, one of which is $L$-{\good} (or weakly $L$-{\good}) while the other is not.

	Take the group $\Z ^2=\langle a, b\;|\;[a, b]\rangle$ with rational structure $L = \{a,a^{-1}, b,b^{-1}\}^*$. The sets $Q_1 := L$ and $Q_2 := \{a^ib^j: i,j \in \Z\}$ both map onto $G$, but $Q_1$ is $L$-{\good} while $Q_2$ is
not.
	However,
 as we shall see in \cref{cor:GoodnessSubmonoidGenInd},
two finitely generated submonoids of 
$(S^\pm)^*$
that map to the same submonoid of $G$ are either both (weakly) $L$-{\good} or
both not (weakly) $L$-\good.
\end{remark}

Recognising dependence on the choice of $Q$, we define $L$-\goodness\
for a subset $K$ of $G$ as follows:
\begin{defn}\label{def:GoodK}
  Given  a rational structure $(G, L)$ for a group $G = \langle S \rangle$ 
and a regular subset $Q$ of  $(S^\pm)^*$ which is  \emph{(weakly-)$L$-{\good}}\ over $G$ with 
constants $k$, $c$, we say that  $K= \mu(Q)$ is \emph{(weakly-)$L$-{\good}}\ with respect to $Q$, with 
constants $k$, $c$.
\mcomment{SR: added `with respect to $Q$', so that we can put that in for clarification when we need it}
\end{defn}
We dedicate the rest of this section to understanding the relationship between weakly-$L$-{\good}, $L$-{\good}\ and $L$-quasi-convex subsets.
It is clear that $L$-{\good}\ subsets are always weakly-$L$-{\good}, 
and we shall prove below that a regular language $Q$ that is weakly-$L$-{\good}\ is also $L$-quasi-convex. Moreover, for a submonoid
given by a finite generating set $T$, if $T^*$ is weakly-$L$-{\good}\ it is also $L$-{\good}\ and, for finitely generated subgroups, the three notions of convexity are equivalent. Briefly, we observe the following equivalences.
\begin{itemize}
\item For a regular subset $Q$ of $(S^\pm)^*$ and rational subset $\mu(Q)$ of $G$,
  \mcomment{JB: reordered these by increasing specificity, and made the implications go left to right. \\
  Changed to refer to $Q$ instead of $\mu(Q)$.}
\begin{align*}
 Q \textrm{~is~}L\text{-\good} &\Longrightarrow  Q \textrm{~is~}\text{weakly-}L\text{-{\good}}\\ &\Longrightarrow   \mu(Q) \textrm{~is~}L\text{-quasi-convex}.
\end{align*}
\item For a finite subset  $T$ of $(S^\pm)^*$ and submonoid $\mu(T^*)$ of $G$, \mcomment{AD: changed from $\mu(T^*)\subseteq G$ to $T$}
\begin{align*}
  T^* \textrm{~is~}L\text{-\good} &\Longleftrightarrow  T^*\textrm{~is~}\text{weakly-}L\text{-{\good}}\\ &\Longrightarrow \mu(T^*) \textrm{~is~} L\text{-quasi-convex}.
\end{align*}
\item For a finite subset $T$ of $(S^\pm)^*$ and subgroup $\mu((T^\pm)^*)$ of $G$, \mcomment{JB: changed $T^*$ to $(T^\pm)^*$ and replaced $\mu((T^\pm)^*) \le G$ by $T$.}
\begin{align*}
  (T^\pm)^*\textrm{~is~}L\text{-\good} &\Longleftrightarrow (T^\pm)^*\textrm{~is~} \text{weakly-}L\text{-{\good}}\\ &\Longleftrightarrow \mu((T^\pm)^*) \textrm{~is~}  L\text{-quasi-convex}.
\end{align*}
\end{itemize}
\mcomment{AD: changed $\pi$ to  $\mu$ above}
\cref{prop:Lgood_is_stronger_rational_subsets}, 
\cref{thm:weaklyLgood_implies_Lgood_submonoids}
and \cref{prop:Lgood_equiv_Lqc_subgroups},
which follow,
provide detail and justification for these observations.
\mcomment{SR: reworded text around the three propositions. In the end I put the proof of the first after its statement, realising that the figures already took up space that broke up the three statements, and the figures were anyway not necessarily placed next to the propositions.}

\begin{prop}\label{prop:Lgood_is_stronger_rational_subsets}
	Let  $G=\langle S\rangle$ be a finitely generated group with rational structure $(G, L)$ and let $Q\subset (S^\pm)^*$ be a regular set. If $Q$ is 
weakly-$L$-\good\ over $G$ with constant $k$ then it is also $L$-quasi-convex with constant $k'=k+|V|$, where $|V|$ denotes the number of vertices in the 
minimal automaton accepting $Q$.
\end{prop}
\begin{proof}
	Let $w$ be a word in $L\cap \mu^{-1}(K)$ and let $w'$ be the 
associated word in $Q$ with which we are provided because of $K$ being 
weakly-$L$-\good. 
We know that, for $0< i< |w|$, each vertex $w_i$ is at distance at most 
$k$ from the vertex $w'_{f(i)}$ in $\Gamma(G, S)$, and if the prefix 
words $w'(f(i))$ were in $Q$ this would already show that $K$ is 
$L$-quasi-convex. Even though this is not necessarily true, 
something very close to this does hold, as follows.
\mcomment{SR: reworded. I hope this is what Lucia intended. And a little minor rewording below.}
 We claim that there exist  
words $u(i)$ for $0< i< |w|$, each 
of length less than $|V|$, so that the concatenation $x(i) = w'(f(i))u(i)$ 
belongs to $Q$ and therefore, if we see  $x(i)$  as a path in 
$\Gamma(G,S)$ starting at $1_G$, $w_i$ is at distance at most $k+|V|$ from the final vertex of $x(i)$. 
	
	To see that such a word $u(i)$ exists, let $\A_Q$ be the minimal, 
trim automaton associated to $Q$; note that each prefix word $w'(f(i))$ labels a path in  $\A_Q$ starting at the initial state and finishing at a certain state $v$ that is not necessarily accepting. Since $\A_Q$ is trim, there is a sequence of transitions in the automaton, starting at $v$ and ending in an accepting state, and since $\A_Q$ has a finite number $|V|$ of states, this sequence of transitions can be chosen to be of length at most $|V|$. We just need to take $u(i)$ to be the word labelled by this sequence of transitions.
\end{proof}

\begin{prop}\label{thm:weaklyLgood_implies_Lgood_submonoids}
  Let $G=\langle S\rangle$ be a finitely generated group and $(G, L)$ a rational structure as above. Let $T$ be a finite subset of $(S^\pm)^*$ such that $T^*$ is weakly $L$-\good\ with constant $k$, and let $M=\mu(T^*)$ be the induced submonoid of $G$.
 Define constants $\ell$ and $c$ by
  \begin{flalign*}
    \ell= &\max\{|m|_{S}\colon \: m\in T \},\\
                 c = &\max\{|m|_{T} \colon \: m\in M \textrm{ and }\: |m|_{S} \leq 2(k + \ell)+1\}.
\end{flalign*}
Then the submonoid $M$ is $L$-\good, with respect to $Q=T^*$ (\cref{def:GoodK}), \mcomment{AD: added wrt $T^*$}  with constants $k'=k+\ell$, $c' = c\ell$.
\end{prop}

\mcomment{JB: Reformulated the statement for clarity. Changed $l$ to $\ell$ and $c_1$ to $c$. The constants of $L$-proximacy are now $k'$ and $c'$. Note that $k'=k+\ell$, not $k$; I hope that doesn't matter. Added proof. Commented out figure for now; it will need adjusting if we want to use it.}

\begin{proof}
  (Illustration in \cref{fig:wLqcu_is_Lqcu}.)
Let $w$ be an $L$-representative of an element of $M$. Since $T^*$ is weakly $L$-\good, there exists $w'$ in $T^*$ such that $w$ asynchronously $k$-fellow travels with $w'$, as described in \cref{def:WeaklyGoodQ}. Let $f$ be the non-decreasing function specified in that definition. Since $w'$ is a concatenation $t_1\cdots t_u$ of words from $T$, we may consider the prefixes $t_1\cdots t_v$ for $v\le u$; call these the $T$-prefixes of $w'$.   
For each $i$ in $[0,|w|]$, define $g(i)$ to be the least number no less than $f(i)$, such that $w'_{g(i)}$ is a $T$-prefix of $w'$. Since $d_S(w_i,w'_{f(i)})\le k$ and $d_S(w'_{f(i)},w'_{g(i)})\le \ell$, we see that $d_S(w_i, w'_{g(i)})$ is bounded above by $k+\ell$. Now 
if $i<|w|$, then $w'_{g(i)}$ is a left divisor in $T^*$ of $w'_{g(i+1)}$ (since $g$ is a non-decreasing function). So there exists $x_{i+1}\in T^*$ such that $w'_{g(i)}x_{i+1} = w'_{g(i+1)}$, and since $d_S(w'_{g(i)},w'_{g(i+1)}) \le 2(k+\ell)+1$, we may assume that $x_{i+1}$ has $T$-length at most $c$. Let $w''$ be the word $x_1\cdots x_{|w|}$ in $T^*$, and let $w''_i$ be the prefix $x_1\cdots x_i$. For each $i$ we have $w'_{g(i)} = w''_i$, and it follows that $w$ asynchronously $(k+\ell)$-fellow travels with $w''$. But since the word $x_{i+1}$ has $T$-length at most $c$, we see that its $S$-length is at most $c\ell$, from which the result follows.
\end{proof}
\begin{figure}[ht]
		\centering
\begin{tikzpicture}[thick,scale=.8,arrowmark/.style 2 args={decoration={markings,mark=at position #1 with \arrow{#2}}}]%
  \tikzstyle{every node}=[circle, draw, fill=blue, color=blue,
                        inner sep=0pt, minimum width=6pt]
 \begin{pgfonlayer}{background}                       
                        \draw[->,>=stealth,thick] (0,0) -- (2.9,1);
                        \draw[->,>=stealth,thick] (3,1)--(5.9,1);
                        \draw[->,>=stealth,thick] (6,1)--(8.9,0);
\end{pgfonlayer}                      
\draw (0,0) ++(-.5,0) node[draw=none,fill=none, color=black] {$1_G$};
\draw (10,0) ++(+1.25,0) node[draw=none,fill=none, color=black] {$w_3=w'_{f(3)}=w'_{g(3)}$};
\draw (1,-2) node[circle, draw, fill=black, color=black, inner sep=0pt, minimum width=3pt] (t1) {};
\draw (4,-2.5) node[circle, draw, fill=black, color=black, inner sep=0pt, minimum width=3pt,label={[label distance=1pt, color=black]0:{$w'_{g(1)}$}}] (t2) {};
\draw (5.5,-1.5) node[circle, draw, fill=black, color=black, inner sep=0pt, minimum width=3pt] (t3) {};
\draw (7.5,-3.5) node[circle, draw, fill=black, color=black, inner sep=0pt, minimum width=3pt,label={[label distance=1pt, color=black]0:{$w'_{g(2)}$}}] (t4) {};
\draw (.7,-1.4) node[draw=none,fill=none, color=black] {\scriptsize{$t_{i_1}$}};
\draw (2.5,-3.3) node[draw=none,fill=none, color=black] {\scriptsize{$t_{i_2}$}};
\draw (5,-1.5) node[draw=none,fill=none, color=black] {\scriptsize{$t_{i_3}$}};
\draw (6.6,-3.6) node[draw=none,fill=none, color=black] {\scriptsize{$t_{i_4}$}};
\draw (8.5,-.9) node[draw=none,fill=none, color=black] {\scriptsize{$t_{i_5}$}};
\foreach \x in {0,3}{
  \draw (\x*3,0) node (\x) {};
}
 \foreach \x in {1,2}{
   \draw (\x*3,1) node[label={[label distance=0pt, color=black]90:{$w_{\x}$}}] (\x) {};
   \draw (\x*4.5-3,-\x*.5-2) node[label={[label distance=-3pt, color=black]250:{$w'_{f(\x)}$}}] (-\x) {};
 }
 \draw (1.5,-2.5) node {};
  \draw[<->,>=stealth, shorten >=.5pt, shorten <=.5pt] (1) -- (-1);
 \draw[<->,>=stealth, shorten >=.5pt, shorten <=.5pt] (2) -- (-2);
  \draw[<->,>=stealth, shorten >=.5pt, shorten <=.5pt] (-1) -- (t2);
  \draw[<->,>=stealth, shorten >=.5pt, shorten <=.5pt] (-2) -- (t4);
 \draw (2.6,-1) node[draw=none,fill=none, color=black]  {\scriptsize{$\le k$}};
 \draw (6.4,-1.7) node[draw=none,fill=none, color=black]  {\scriptsize{$\le k$}};
 \draw (4,0) node[draw=none,fill=none, color=black]  {\scriptsize{$\le k+l$}};
 \draw (7,0) node[draw=none,fill=none, color=black]  {\scriptsize{$\le k+l$}};
 \draw (2.7,-2.2) node[draw=none,fill=none, color=black]  {\scriptsize{$\le l$}};
 \draw (6.75,-2.9) node[draw=none,fill=none, color=black]  {\scriptsize{$\le l$}};
  \begin{pgfonlayer}{background}     
 \draw[->,>=stealth,thick] (0) to[out=330, in=90, looseness=1.5] (t1);
 \draw[->,>=stealth,thick] (t1) to[out=315, in=210, looseness=1.4] (t2);
 \draw[->,>=stealth,thick] (t2) to[out=80, in=135, looseness=1.4] (t3);
 \draw[->,>=stealth,thick] (t3) to[out=270, in=180, looseness=1.2] (t4);
 \draw[->,>=stealth,thick] (t4) to[out=90, in=220, looseness=1.4] (3);
 \draw[->,>=stealth,thick,dashed,shorten >=.5pt] (0) to [out=245, in=235, looseness=2] (t2);
 \draw[->,>=stealth,thick,dashed,shorten >=.5pt] (t2) to [out=300, in=225, looseness=1.4] (t4);
 \draw[->,>=stealth,thick,dashed,shorten >=.5pt] (t4) to [out=45, in=270, looseness=1.4] (3);
 \end{pgfonlayer}
 \draw (1.5,-4.5) node[draw=none,fill=none, color=black]  {$x_1$};
 \draw (6,-4.7) node[draw=none,fill=none, color=black]  {$x_2$};
 \draw (9.1,-2.5) node[draw=none,fill=none, color=black]  {$x_3$};
\draw[<->,>=stealth, shorten >=.5pt, shorten <=.5pt] (1) -- (t2);
\draw[<->,>=stealth, shorten >=.5pt, shorten <=.5pt] (2) -- (t4);
 \draw[thick,->,>=stealth,decorate,decoration={bent,amplitude=15mm}] (.5,1.2) -- (8.5,1.2);
 \draw (4.5,3) node[draw=none,fill=none, color=black]  {$w$};
 \draw (4.5,-5.5) node[draw=none,fill=none, color=black]  {$w'=t_{i_1}t_{i_2}t_{i_3}t_{i_4}t_{i_5}$};
 \draw (4.5,-6) node[draw=none,fill=none, color=black]  {$w''=x_1x_2x_3$};
\end{tikzpicture}
\caption{The paths defined in the proof of \cref{thm:weaklyLgood_implies_Lgood_submonoids}.
}
\label{fig:wLqcu_is_Lqcu}
\end{figure}
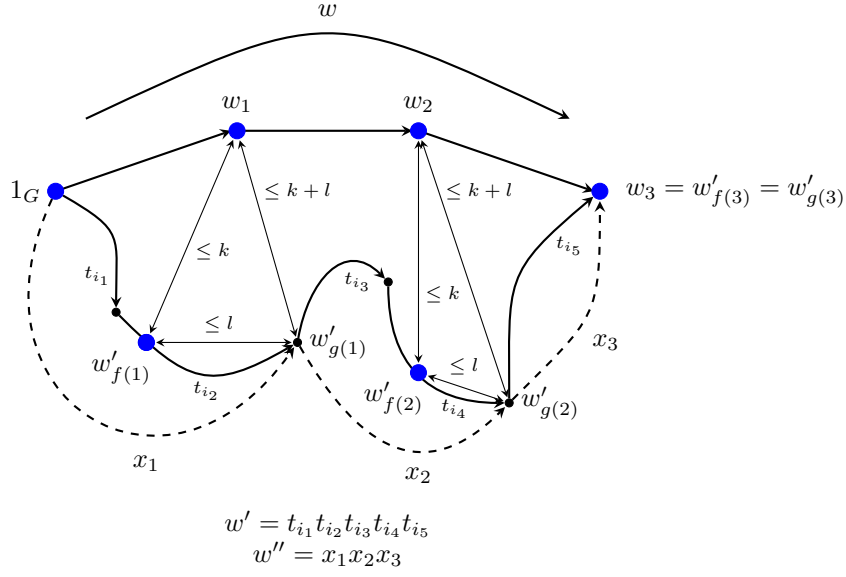

\begin{prop}\label{prop:Lgood_equiv_Lqc_subgroups} Let $G=\langle S\rangle$ be a finitely generated group and $(G, L)$ a rational structure.
Let $T$ be a finite subset of $(S^\pm)^*$, and let $H$ be the subgroup of $G$ generated by $T$. \mcomment{AD: replaced $\mu(T)$ by $T$} 
Define constants $\ell$ and $c$ by
	\begin{flalign*}
			\ell = &\max\{ |h|_{S} \colon \: h\in T \}, \\
		c = &\max\{ |h|_{T} \colon \: h\in H \textrm{ and } |h|_{S} \leq 2k+1\}.
	\end{flalign*}\mcomment{JB: reformulated, and changed $c_1$ and $c_2$ to $c$ and $\ell$, for consistency with \cref{thm:weaklyLgood_implies_Lgood_submonoids}.}
	If $H$ is $L$-quasi-convex with constant $k$, then
        \mcomment{JB: changed $T^*$ to $(T^\pm)^*$.}
$H$ is $L$-\good\ with respect to $Q=(T^\pm)^*$, with constants $k$ and $c'=c\ell$. \mcomment{AD: put statement and proof in terms of \cref{def:GoodK} and \cref{sec:cayley}}
\end{prop}

\begin{proof}
Let $w$ be an $L$-representative of an element of $H$. For $j\le |w|$ (following the conventions of \cref{sec:cayley}) let $w(j)$ be the prefix of $w$ of length $j$. Then $w_j=\mu(w(j))$ lies in the $k$-neighbourhood of $H$ in $\Gamma(G,S)$. Let $h_j$ be an element of $H$ such that
$d_S(w_j,h_j)\le k$, with $h_0=1$ and $h_{|w|}=\mu(w)$.  For $1\le j\le |w|$, define $t_j$ to be a word of minimal $T$-length in $(T^{\pm})^*$ such that $\mu(t_j)=h_{j-1}^{-1}h_j$.  
Then setting $q_j= t_1 \cdots t_j$ for each $j$ (with $q_0=\epsilon$), we see that $\mu(q_j)=h_j$ for all $j$. Since $d_S(h_j,h_{j+1})\le 2k+1$, we see that $|h_j^{-1}h_{j+1}|_S\le 2k+1$, and hence that $|h_j^{-1}h_{j+1}|_T\le c$. Since the elements $t_j$ have minimal $T$-length by assumption, it follows that $|t_j|_S\le c\ell$ for all $j$. So we see that
$w$ $k$-fellow travels with $q_{|w|}$ up to at most $c'$-reparametrization, and so $H$ is $L$-{\good} with constants $k$ and $c'$. \mcomment{JB: added this proof. Commented out the figure for now; I don't quite see how to adjust it. AD: added $t$'s and removed $\le$'s, and put the figure back.}
\end{proof}

	\begin{figure}[ht]
		\centering
\begin{tikzpicture}[thick,scale=.8,arrowmark/.style 2 args={decoration={markings,mark=at position #1 with \arrow{#2}}}]%
  \tikzstyle{every node}=[circle, draw, fill=blue, color=blue,
                        inner sep=0pt, minimum width=6pt]
 \begin{pgfonlayer}{background}                       
                        \draw[->,>=stealth,thick] (0,0) -- (2.9,1);
                        \draw[->,>=stealth,thick] (3,1)--(5.9,1);
                        \draw[->,>=stealth,thick] (6,1)--(8.9,0);
                        \draw[->,>=stealth,thick,decorate,decoration={snake, segment length=5mm, amplitude=.5mm}] (0,0)--(2.9,-2.9);
                        \draw[->,>=stealth,thick,decorate,decoration={snake, segment length=5mm, amplitude=.5mm}] (3,-3)--(5.8,-3);                                               \draw[->,>=stealth,thick,decorate,decoration={snake, segment length=5mm, amplitude=.5mm}] (6,-3)--(8.9,-.1);
                        
\end{pgfonlayer}                      
\draw (0,0) ++(-1,0) node[draw=none,fill=none, color=black] {$w_0=h_0$};
\draw (9,0) ++(+1.25,0) node[draw=none,fill=none, color=black] {$w_3=h_3$};
\draw (2.2,-2.5)  node[draw=none,fill=none, color=black] {\small{$t_1$}};
\draw (5,-3.3) node[draw=none,fill=none, color=black] {\small{$t_2$}};
\draw (8.5,-1)  node[draw=none,fill=none, color=black] {\small{$t_3$}};
\foreach \x in {0,3}{
  \draw (\x*3,0) node (\x) {};
 }
 \foreach \x in {1,2}{
   \draw (\x*3,1) node (\x) {};
   \draw (\x*3,-3) node (-\x) {};
   \draw (\x*3,0) ++(0,1.5) node[draw=none,fill=none, color=black] {$w_{\x}$};
   \draw (\x*3,-3) ++(0,-.5) node[draw=none,fill=none, color=black] {$h_{\x}$};
 }
 \draw[<->,>=stealth] (3,.75) -- (3,-2.75);
 \draw[<->,>=stealth] (6,.75) -- (6,-2.75);
 \draw (3.5,-1) node[draw=none,fill=none, color=black]  {$\le k$};
 \draw (6.5,-1) node[draw=none,fill=none, color=black]  {$\le k$};
 \draw[thick,->,>=stealth,decorate,decoration={bent,amplitude=15mm}] (.5,1.2) -- (8.5,1.2);
 \draw (4.5,3) node[draw=none,fill=none, color=black]  {$w$};
 %
\end{tikzpicture}                
\caption{The paths defined in \cref{prop:Lgood_equiv_Lqc_subgroups}}
		\label{fig:Lqc_generalisation}
              \end{figure}
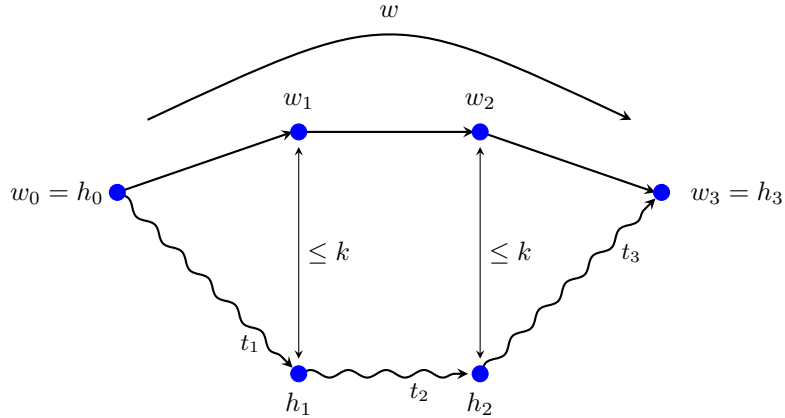

At the end of \cref{sec:Notation} we pointed out that a subgroup is $L$-quasi-convex if and only if it is $L$-recognisable. Similarly, we have the following.
\begin{prop}\label{prop:goodrec}
  Given a group $G=\langle S\rangle$ and a rational structure $(G, L)$, where $L\subseteq (S^\pm)^*$,
  a rational subset $K$ of $G$ is $L$-\good, with respect to some \mcomment{AD: added more detail, following comments of L and J}
  regular subset of $(S^\pm)^*$, if and only if it is $L$-recognisable.
\end{prop}
One direction is immediate, as $L$-recognisability means that $L\cap \mu^{-1}(K)$ is a regular language $Q$, and $Q$ is by definition $L$-\good\ with constants $k = 0$, $c = 1$. The other direction of the statement will be proved 
constructively 
in \cref{thm:algorithm_terminates}. As a consequence of this result, 
\cref{prop:Lgood_is_stronger_rational_subsets} and \cref{thm:weaklyLgood_implies_Lgood_submonoids}
 we have the following. \mcomment{AD: added this which appears to change the next proposition to an iff statement. }
\mcomment{SR: replaced `we have:' by `we have the following'}
\begin{prop}\label{proposition:SubmonoidGoodRecognisable}
  Let  $G=\langle S\rangle$ be a group with rational structure $(G, L)$,   where $L\subseteq (S^\pm)^*$ and let $T$ be a finite subset
  of $(S^\pm)^*$.
  Then the submonoid 
$\mu(T^*)$ of $G$ is weakly-$L$-\good, with respect to some regular subset of $(S^\pm)^*$, 
if and only if 
it is $L$-recognisable.
\mcomment{AD: and again: added more detail, following comments of L and J}
\end{prop}
\begin{qu}
  Is it possible to replace ``with respect to some regular subset of $(S^\pm)^*$'' by ``with respect to $T^*$'' in the Lemma above? It
  follows from \cref{prop:goodrec} that 
  this works for the only if part of the statement but we do not know whether the property that $\mu(T^*)$ is $L$-recognisable
 and weakly-$L$-\good, with respect to some regular set $Q$, is sufficient to ensure that it is weakly-$L$-\good\ with respect to $T^*$.
\end{qu}
Before we move on to the next section, we give an equivalent definition for weakly-$L$-\good\ submonoids, which clarifies the similarity to classical $L$-quasi-convexity and the independence on generating set.
\begin{prop}\label{prop:AltDefGoodSubmonoid}
	Given a finitely generated group $G=\langle S\rangle$ with rational structure $(G, L)$ and a finite set $T\subset (S^\pm)^*$, the language $T^*$ (and therefore the submonoid $M = \mu(T^*)$ of $G$) is weakly-$L$-\good\ 
if and only if there is a constant $k$ such that for every 
$w\in L\cap \mu^{-1}(M)$ there are elements $m_1, \ldots ,m_{|w|}$ in $M$ such that,  in $\Gamma(G,S)$,  $m_i$ is at distance at most $k$ from the vertex $w_i$, $m_{|w|}= w$ and $m_{i+1}\in m_iM$.
\end{prop}
\mcomment{AD: (sketch of) proof added}
\begin{proof} Recall that $w=w_{|w|}$. If such elements  $m_1, \ldots ,m_{|w|}$ exist, let $n_1$ be a $T^*$ representative of $m_1$, and
  let $n_{j+1}$ be a word in $T^*$ such that $m_{j+1}=_M m_jn_{j+1}$, for $1\le j\le |w|-1$; which exist as $m_{j+1}\in m_jM$.  Then 
  the word $n_1n_2\cdots n_{i}$ is in $T^*$ and is mapped by $\mu$ to $m_i$, for $1\le i\le |w|$, whence $w=_M n_1n_2\cdots n_{|w|}$ satisfies the
  weak-$L$-\good\ condition, with constant $k$.

  Conversely, if $G$ is weakly-$L$-\good,  with constant $k$, let $l := \max\{|t|_{S}\, :\, t\in T\}$ and let $k'=k+l$.
  If  $w'$ is the word corresponding to $w$, from \cref{def:WeaklyGoodQ}, then using
  the argument given in the sketch of the proof of \cref{thm:weaklyLgood_implies_Lgood_submonoids}, it follows that  elements   $m_1, \ldots ,m_{|w|}$ of $M$ exist, such that
  the distance between $w_i$ and $m_i$ is at most $k'$. 
\end{proof}
The following is an immediate consequence of \cref{prop:AltDefGoodSubmonoid}.
\begin{cor}\label{cor:GoodnessSubmonoidGenInd}
	In a group with rational structure $(G, L)$, 
 if the images in $G$ of finite subsets $T$ and $U$ of $(S^\pm)^*$
generate the same submonoid of $G$, then $T^*$ is (weakly) $L$-{\good} precisely when
$U^*$ is.
\end{cor}

\section{A folding procedure for rational subsets}\label{sec:Construction} 

Throughout \mcomment{AD: changed the section heading} this section, we adopt the convention that a finitely presented group is given by a finite presentation $G=\langle S \:| \: R\rangle$
for which the set $R$ of relators is closed under taking inverses and cyclic permutation of words.\mcomment{AD: is it also necessary to assume that in a regular structure $(G,L)$, the language $L$ consists of freely reduced words?
  In any case, care needs to be taken to use $\mu$ and $\pi$ correctly. }

\subsection{Rewriting}
In the next paragraphs we define a finite rewriting system, parallel and similar to the one presented in \cite[$\S4$]{kmw17}, whose aim is to reflect the behaviour of the algorithm that will be introduced later in the section in \cref{construction:completion}. 

Let $G=\langle S \:| \: R\rangle$ be a finitely presented group and $\widetilde D_R$,  $\widetilde D_{\red}$ be the following rewriting systems on $(S^\pm)^*$, where $R_0=\{ss^{-1}\,|\, s\in S^\pm\}$:
\begin{flalign*}
  \widetilde D_R    &=  \big\{\emptyword\rightarrow t\, |\,  t \in R \cup R_0\},	\\
	\widetilde D_{\red} &=  \big\{r\rightarrow \emptyword\, |\,  r\in R_0\big\} .
\end{flalign*}
Given words $w, w'$ in $(S^\pm)^*$, we say that $w'$ is one $\widetilde D_R$-rewriting step away from $w$  if  there exists an integer $n\ge 1$ and 
words $u^{(i)}$ such that $w$ can be written as the concatenation $$w = u^{(1)}u^{(2)}\ldots u^{(n)},$$ (where some of the $u^{(i)}$ may be the empty word) 
and 
$$w' = u^{(1)}t^{(1)}u^{(2)}\ldots t^{(n-1)} u^{(n)}$$ so that, for each $i$, $\emptyword\rightarrow t^{(i)}$ is a rule in $\widetilde D_R$. We denote this by $w\xrightarrow[R]{}w'$. Note that, by allowing $n = 1$ in the definition of a $\widetilde D_R$-rewriting step, we have that $w\xrightarrow[R]{}w$, and by allowing $u^{(i)}=\emptyword$  we allow insertion of multiple relations in the
same position in $w$ and at its
beginning or end. If there are words $w_{(1)},\ldots , w_{(k-1)}$ such that $$w\xrightarrow[R]{}w_{(1)}\xrightarrow[R]{}\ldots  \xrightarrow[R]{}w_{(k-1)}\xrightarrow[R]{}w',$$
then we say that $w$ is $k$ $\widetilde D_R$-rewriting steps away from
$w'$ and we write $w\xrightarrow[R]{k}w'$. We write $w\xrightarrow[R]{*}w'$ to describe that $w'$ is obtained from $w$ in 0 or more $\widetilde D_R$-rewriting steps.
Using the rewriting system $\widetilde{D}_{\red}$, we similarly define $\widetilde D_{\red}$-rewriting steps together with the notation $$w\xrightarrow[\red]{}w'\;, w\xrightarrow[\red]{k}w' \;\text{ and }\; w\xrightarrow[\red]{*}w'.$$
\begin{lemma}\label{lemma:sameg_kstepsaway}
  If $w,u$ and $w'$ are elements of $(S^\pm)^*$ such that  
  $w\xrightarrow[\red]{} u\xrightarrow[R]{} w'$, then there exists a word $v\in (S^\pm)^*$  such that $w\xrightarrow[R]{} v\xrightarrow[\red]{} w'$.
\end{lemma}
Rather than proving this in general we exhibit an example. The (rather tedious) general proof follows the  same lines. 
  For example, suppose $w=ass^{-1}btt^{-1}c$, where $s,t\in S$ and $a,b,c\in (S^\pm)^*$, with $u=abc$ and then that $u=defg$, where $d,e,f,g\in (S^\pm)^*$, and 
  $w'=d\alpha e\beta f\gamma g$, where $\alpha,\beta,\gamma\in R\cup \{ss^{-1}\,|\, s\in S\}$. Then we may refine the factors of
  $u$ to allow either factorisation. If, for instance, $a=a_1a_2$ and $b=b_1b_2$ where $d=a_1$, $e=a_2b_1$, $f=b_2$ and $g=c$ then
  \[u=(a_1a_2)(b_1b_2)c=a_1(a_2b_1)b_2c,\] \[w=(a_1a_2)ss^{-1}(b_1b_2)tt^{-1}c=a_1(a_2ss^{-1}b_1)(b_2tt^{-1})c\] and \[w'=a_1\alpha (a_2b_1)\beta b_2 \gamma c.\]
  Defining $v=a_1\alpha a_2ss^{-1}b_1 \beta b_2 tt^{-1}\gamma c$ we then 
have $w \xrightarrow[R]{}  v\xrightarrow[\red]{} w''$, as required. 

As a consequence of \cref{lemma:sameg_kstepsaway}: 
\begin{lemma}\label{lemma:sameg_DRfollowedbyDred}
	If $w$ and $w'$ are words representing the same element in $G$, then $w'$ can be obtained from $w$ by applying a sequence of $\widetilde D_R$ rewrites followed by $\widetilde D_{\red}$ rewrites.
\end{lemma}

\begin{defn}
	Given a finitely presented group $G=\langle S|R\rangle$, and  words $w, w'$ in $(S^\pm)^*$ we say that $w'$ is one step away from $w$ if there exists an intermediate word $w''$ such that  $$w\xrightarrow[R]{} w''\xrightarrow[\red]{*} w'.$$ We say that $w'$ is $k$ steps away from $w$ if there are words $w_{(0)}= w, w_{(1)},\ldots , w_{(k)}=w'$ such that $w_{(i)}$ is one step away from $w_{(i-1)}$ for $i = 1,\ldots,k$.
\end{defn}

\begin{lemma}\label{lemma:k_steps_criterion}
  Let $w$ and  $w'$ be words in $(S^\pm)^*$. Then $w'$ is $k$ steps away from $w$ if and only if there exists an intermediate word $w''$ such that  $$w\xrightarrow[R]{k} w''\xrightarrow[\red]{*} w'.$$
\end{lemma}
\begin{proof}
  If $w'$ is $k$ steps away from $w$ then, by repeatedly applying the result of \cref{lemma:sameg_kstepsaway} a rewriting of the required
  form may be obtained. The converse follows directly from the definition.  
\end{proof}
\mcomment{SR: revised \cref{subsec:folding}, in particular deleting detail repeated from earlier on}

\begin{cor} \label{corollary:steps}
If $w$ and $w'$ are words representing the same element in $G$, then there exists $k$ such that $w'$ is $k$ steps away from $w$.
\end{cor}
\mcomment{AD: The converse of this corollary obviously also holds. Is it worth making into an iff statement or is it best left as it is?}
\begin{proof}
This follows immediately from \cref{lemma:sameg_DRfollowedbyDred} and \cref{lemma:k_steps_criterion}. 
\end{proof}

\subsection{The procedure}\label{subsec:folding}
Now, we describe the procedure that will allow us to recognise all of the 
$L$-representatives of a rational subset $K$ of a group $G=\langle S \rangle$, starting from an automaton that recognises a rational subset $Q$ of $(S^\pm)^*$ with $\mu(Q)=K$.

\mcomment{SR: inserted formal definition of folding}
A key step in our procedure \emph{folds} a finite state automaton
 $\A=(S^{\pm}, V, E, A, v_0)$  to produce a finite state automaton $\A'$
whose language is the \emph{folding} of $L(A)$.

\begin{definition}\label{def:folding}
The \emph{folding} of a subset $L$ of $(S^\pm)^*$
is the smallest subset $L'$ of $(S^\pm)^*$ containing $L$ for which,
if $uxx^{-1}v \in L'$ then $uv \in L'$.
\end{definition}

By Benois' theorem \cite{benois,bs21}, the folding $L'$ of a regular language
$L=L(\A)$  is regular, 
and an automaton $\A'$ to recognise it can be constructed using an algorithm 
described in \cite{bs21}. \mcomment{AD: added description of Benois' algorithm}
As it forms a central part of our procedure we briefly outline the
operation of this algorithm. Given a finite state automaton
$M=(S^{\pm}\cup \{\varepsilon\}, V, E, A, v_0)$,
which may have $\varepsilon$ transitions,
the algorithm constructs a new automaton $M'$, as follows. Regard $M$
as a directed, labelled graph. If $p$ and $q$ are vertices
of $M$ and, for some $a\in S^{\pm}$, there is a path in $M$ from $p$ to $q$ with label $aa^{-1}$, and no path from $p$ to $q$ with label equal to  $1$ in the
free monoid $(S^{\pm})^*$,
then an edge from $p$ to $q$, labelled $\varepsilon$, is added to $M$. This
is repeated over all pairs $(p,q)\in V\times V$ to give an automaton $M_1$.
The entire process is then repeated, starting with $M_1$ in place of $M$, to
construct $M_2$. As no
new vertices are added and the
only new edges added have label $\varepsilon$, eventually the process stabilises
with $M_n=M_{n+1}=M'$; at which point $L(M')=L(M)'$ and we call $M'$ the
\emph{folding} of $M$.
For details see \cite{bs21}. We call this algorithm \emph{Benois' algorithm}. 
%

\begin{construction}\label{construction:completion}
Let $G=\langle S \:|\:  R\rangle$ be a finitely presented group, with $R$ closed under taking inverses, 
and let $\A_0$ be an automaton over the alphabet $ S^\pm$ recognising a
 regular language $Q \subseteq  (S^\pm) ^*$.
For each $n>0$, we construct $\A_n$ from $\A_{n-1}$ by performing the following sequence of two steps:

\begin{enumerate}[label={(\arabic*)}]
  \item \label{enum:construction1}
 At each vertex of $\A_{n-1}$ add a cycle corresponding to each relator of
 $R$ as well as a cycle of length two corresponding to each $ss^{-1}$  for  
$s\in S^\pm$, and hence obtain an automaton $\A_n^0$.
  \item Use Benois' algorithm to construct an automaton $\A_n^1$ by  folding  $\A_n^0$.
  \item Determinise $\A_n^1$ to obtain $A_n$.     \mcomment{AD: added step (3)}           
	\end{enumerate}
\end{construction}

\begin{prop}\label{prop:LanguageOfAn}
	A word $w\in (S^\pm)^*$ is recognised by the automaton $\A_n$ if and only if $w$ is at most $n$ steps away from a word accepted by $\A_0$ .
	
\end{prop}

\begin{proof}
  By definition, $\A_{n}$ is the automaton obtained by folding and determinising $\A^0_{n}$, which means  that the language accepted by $\A_{n}$ is the set of all words that can be obtained through free reductions from words accepted by $\A^0_{n}$. Consequently, a word $w'$ is accepted by $\A_{n}$ if and only if there is a word $w''$ accepted by $\A^0_{n}$ such that $w''\xrightarrow[\red]{*}w'$.
	
	Take now an arbitrary word $w''$ accepted by $\A^0_{n}$. It is clear from the construction that we can write it as $$w'' = r^{(0)}w_1r^{(1)}w_2\dots w_m r^{(m)},$$ where each word $r^{(i)}$ is in the language $ (R\cup  R_0)^*$ (where $R_0=\{ss^{-1} \mid s\in S^\pm \}$, as above) and $w=w_1w_2\dots w_m$ is a word of length $m$ accepted by $\A_{n-1}$. This means exactly that $w''$ is at most one $R$-rewriting step away from $w$, so altogether we deduce that $$w\xrightarrow[R]{} w''\xrightarrow[\red]{*} w'.$$
	
	This means that $w'$ in $L(\A_{n})$ is one step away from a word $w$ in $L(\A_{n-1})$ and, by induction, we have that the words accepted by $\A_{n}$ are at most $n$ rewriting steps away from those accepted by $\A_0$. 
	
	The converse also follows by induction from the definition of the rewriting systems, since $\widetilde D_R$ emulates step (1) and $\widetilde D_{\red}$ emulates step (2) in \cref{construction:completion}:
	if a word $w'$ is at most one step away from a word $w$ accepted by $\A_{n-1}$ it means that we can find a word $w''$ (possibly equal to $w$) such that $$w\xrightarrow[R]{} w''\xrightarrow[\red]{*} w'.$$ 
Applying \cref{construction:completion}, we see that the word $w''$ is accepted by $\A_{n}^0$ and so the word $w'$ is accepted by $\A_{n}$. 
\end{proof}

\begin{cor}
	Let $\langle S|R\rangle$ be a finitely presented group and $Q_0\subseteq  (S^\pm)^*$ the regular language accepted by some automaton $\A_0$. Denote by $Q_n$, $n\in \mathbb{N}$, the language accepted by the automaton $\A_n$ whose construction is described above.
	Then, $\bigcup_{n\in \mathbb{N}}Q_n = \mupi^{-1}(\mupi(Q_0)).$
\end{cor}
\begin{proof}
	From \cref{prop:LanguageOfAn}, we see that $\mupi(Q_n) = \mupi(Q_0)$ for all $n\in \mathbb{N}$, and it is clear that for any $w\in \mupi^{-1}(\mupi(Q_0))$, there exists $n\in \mathbb{N}$ such that $w$ can be obtained in $n$ steps from a word of $Q_0$. 
\end{proof}

As stated above, for an $L$-\good\ rational subset $K$, our procedure will eventually construct an automaton which recognises all $L$-representatives of $K$. 
\begin{thm}\label{thm:algorithm_terminates}
 Let $(G, L)$ be a rational structure for a finitely presented group $G = \langle S|R\rangle $ and let $Q_0\subseteq (S^\pm)^*$ be the regular language accepted by an automaton $\A_0$. Assume that $Q_0$ is  $L$-\good\ with constants $c$, $k$ (see \cref{def:GoodQ}) and let $K = \mupi(Q_0)$. Then, there exists $n\in \mathbb{N}$ such that the full set of $L$-representatives of $K$ is contained in  $Q_n$, the language accepted by the automaton $\A_n$ whose construction is described above. In particular, $K$ is $L$-recognisable. \mcomment{AD: added the final statement to cover the comment made in the paragraph following \cref{prop:goodrec}}
\end{thm}

The following lemma, which follows easily from the definitions, will aid the proof of \cref{thm:algorithm_terminates}.
\begin{lemma}\label{lemma:ConcatenationRewritingSteps}
	Let $G=\langle S \:|\:  R\rangle$ be a finitely presented group, let $m, n \in \mathbb{N}$ be strictly positive integers, and let $w_1, \dots, w_m, w'_1, \dots, w'_m \in (S^\pm)^*$ be words such that each $w'_i$ can be obtained in at most $n$ steps from $w_i$.
	Then, $w'_1w'_2 \dots w'_m$ can be obtained in $n$ steps from $w_1 w_2 \dots w_m$.
\end{lemma}

\begin{proof}[Proof of \cref{thm:algorithm_terminates}]
	\cref{corollary:steps} tells us that whenever $u$ and $v$ are words in $ (S^\pm)^*$ with $u=_Gv$, there is some minimal $k_{u,v}\in \mathbb{N}$ such that $u$ is $k_{u,v}$ steps from $v$. Define $n\in \mathbb{N}$ by 
	\[
		n =\max\{k_{u,v} \mid  u=_Gv,\ 
		|u|\leq 2k+1,\ |v|\leq c\}
	\]
This maximum exists since the elements $u$ and $v$ under consideration lie in a bounded ball in the Cayley graph of $G$ with respect to $S$. We will show that the  set of $L$-representatives $J= L\cap \mupi^{-1}(K)$ is a subset of $Q_n$.
	
	Let $w$ be an arbitrary word in the language $J$.
	By assumption, there exists a word $w'\in Q_0$ with $w=_Gw'$ such that the paths in the Cayley graph of $G$ corresponding to the words $w$ and $w'$ $k$-fellow travel, up to $c$-reparametrisation. Let $f\colon [0, |w|] \rightarrow [0, |w'|]$ be the reparametrisation map so that, for the paths labelled by $w$ and $w'$, following the notation defined in \cref{sec:Notation},  we have
        $0\le f(i+1)-f(i)\le c$ and $d(w_i, w'_{f(i)}) \leq k$ for all $0\leq i \leq |w|$.
	
	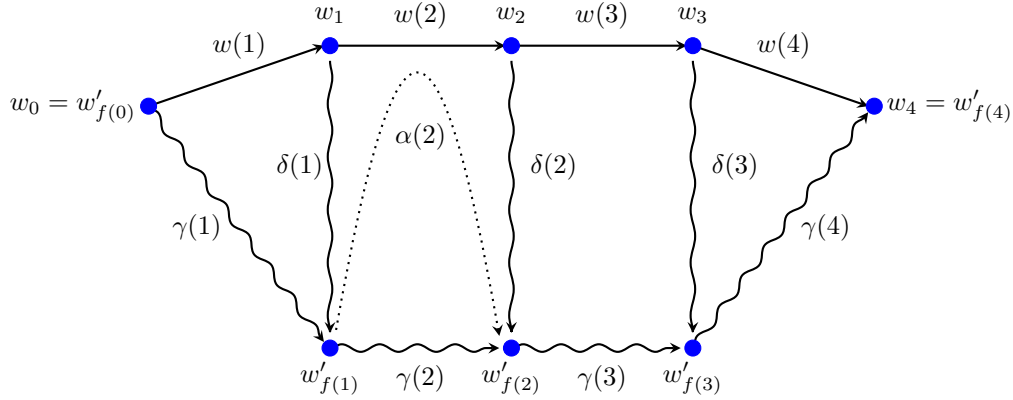
\begin{figure}[ht]
		\centering
\begin{tikzpicture}[thick,scale=.8,arrowmark/.style 2 args={decoration={markings,mark=at position #1 with \arrow{#2}}}]%
  \tikzstyle{every node}=[circle, draw, fill=blue, color=blue,
                        inner sep=0pt, minimum width=6pt]
 \begin{pgfonlayer}{background}                       
                        \draw[->,>=stealth,thick] (0,0) -- (2.9,1);
                        \draw[->,>=stealth,thick] (3,1)--(5.9,1);
                        \draw[->,>=stealth,thick] (6,1)--(8.9,1);
                        \draw[->,>=stealth,thick] (9,1)--(11.9,0);
                        \draw[->,>=stealth,thick,decorate,decoration={snake, segment length=5mm, amplitude=.5mm}] (0,0)--(2.9,-3.9);
                        \draw[->,>=stealth,thick,decorate,decoration={snake, segment length=5mm, amplitude=.5mm}] (3,-4)--(5.8,-4);
                        \draw[->,>=stealth,thick,decorate,decoration={snake, segment length=5mm, amplitude=.5mm}] (6,-4)--(8.8,-4);  
                        \draw[->,>=stealth,thick,decorate,decoration={snake, segment length=5mm, amplitude=.5mm}] (9,-4)--(11.9,-.1);
\end{pgfonlayer}                      
\draw (0,0) ++(-1.25,0) node[draw=none,fill=none, color=black] {$w_0=w'_{f(0)}$};
\draw (12,0) ++(+1.25,0) node[draw=none,fill=none, color=black] {$w_4=w'_{f(4)}$};
\foreach \x in {0,4}{
  \draw (\x*3,0) node (\x) {};
  \begin{pgfonlayer}{background}  
    \draw (\x*3,0) node[draw=none,fill=none, color=white, inner sep=0pt, minimum width=11pt] (B\x) {};
   \end{pgfonlayer}  
 }
 \foreach \x in {1,2,3}{
   \draw (\x*3,1) node (\x) {};
   \draw (\x*3,-4) node (-\x) {};
   \begin{pgfonlayer}{background}  
     \draw (\x*3,1) node[draw=none,fill=none, color=white, inner sep=0pt, minimum width=11pt] (B\x) {};
     \draw (\x*3,-4) node[draw=none,fill=none, color=white, inner sep=0pt, minimum width=11pt] (B-\x) {};
   \end{pgfonlayer}  
   
   \draw (\x*3,0) ++(0,1.5) node[draw=none,fill=none, color=black] {$w_{\x}$};
   \draw (\x*3,-4) ++(0,-.5) node[draw=none,fill=none, color=black] {$w'_{f(\x)}$};

   \draw[->,>=stealth,thick,decorate,decoration={snake, segment length=7mm, amplitude=.3mm}] (B\x) -- (B-\x);
 }

\draw[->,>=stealth,thick,dotted] (B-1) +(.1,.3) to[out=80, in=100, looseness=5.5] (5.8,-3.8);
\draw (.8,-2) node[draw=none,fill=none, color=black]  {$\gamma(1)$};
\draw (4.5,-4.5) node[draw=none,fill=none, color=black]  {$\gamma(2)$};
\draw (7.5,-4.5) node[draw=none,fill=none, color=black]  {$\gamma(3)$};
\draw (11.2,-2) node[draw=none,fill=none, color=black]  {$\gamma(4)$};

\draw (2.5,-1) node[draw=none,fill=none, color=black]  {$\delta(1)$};
\draw (6.7,-1) node[draw=none,fill=none, color=black]  {$\delta(2)$};
\draw (9.7,-1) node[draw=none,fill=none, color=black]  {$\delta(3)$};

\draw (1.5,1) node[draw=none,fill=none, color=black]  {$w(1)$};
\draw (4.5,1.5) node[draw=none,fill=none, color=black]  {$w(2)$};
\draw (7.5,1.5) node[draw=none,fill=none, color=black]  {$w(3)$};
\draw (10.5,1) node[draw=none,fill=none, color=black]  {$w(4)$};

\draw (4.5,-.5) node[draw=none,fill=none, color=black]  {$\alpha(2)$};
\end{tikzpicture}
                \caption{Notation for the proof of \cref{thm:algorithm_terminates}}
		\label{fig:proof_qcu}
	\end{figure}
	See \cref{fig:proof_qcu} for an illustration of the following  paragraph. Following \cref{sec:Notation} denote by $w(1),\ldots,$ $w(|w|)$ the letters in $w$, which we can think of as $S$-labelled, 1-edge subpaths of $w$ in the Cayley graph. 
	Let $\gamma(1), \gamma(2),  \ldots, \gamma(|w|)$ be the subwords of $w'$ defined by respectively  taking its first $f(1)$ letters, then the next $f(2)-f(1)$ letters, etc. until finally making $\gamma(|w|)$ out of the last $f(|w|)-f(|w|-1)$ letters of $w'$. Each $\gamma(i)$ can be thought of as a subpath of $w'$ in the Cayley graph  of length less than or equal to $c$.
	Now, let $\delta(i)\in (S^\pm)^*$ be a word representing a geodesic path between the vertices $w_i$ and $w_{f(i)}$, for $i = 0, \ldots, |w|$. Note that such words have length less or equal than $k$, and that $\delta(0), \delta(|w|)$ are trivial paths.
	Lastly, define $\alpha(i)$ to be the path in the Cayley graph (and the word in $(S^\pm)^*$) defined by concatenating the paths $\delta^{-1}(i-1)\cdot w(i)\cdot\delta(i)$, where $\delta^{-1}$ denotes the natural inverse path. The paths $\alpha_1,\ldots,\alpha_{|w|} $ all have length less than or equal to $2k+1$. 
	
	Note that each of the paths $\alpha(i)$ defines a word equivalent  to $\gamma(i)$ in $G$, and because $|\alpha(i)|\le 2k+1$ and
        $\gamma(i)\le c$ the word $\alpha(i)$ is at most $n$ steps away from $\gamma(i)$.  By  \cref{lemma:ConcatenationRewritingSteps} we conclude that the word defined by $\alpha = \alpha(1)\cdot \ldots \cdot \alpha(|w|)$ is at most $n$ steps away from $w' = \gamma(1)\cdot \ldots \cdot \gamma(|w|)$  and, since $w'$ is in $Q_0$, this means that $Q_n$ contains the word defined by $\alpha$.
	Furthermore,  $\alpha=\alpha(1)\cdot \ldots \cdot \alpha(|w|)$  is freely reducible to $w = w(1)\cdot \ldots \cdot w(|w|)$ which means that $w$ is also accepted by $Q_n$ (since its language is closed under free reductions).

        For the final statement; since the intersection of two regular languages is regular,
\mcomment{AD: added proof of final statement}        it follows that $\mu^{-1}(K)\cap L=Q_n\cap L$ is regular, so $K$ is $L$-recognisable.
\end{proof}

\begin{definition}\label{def:flower}
Given a set $T$ of words over $S^\pm$, we define the \emph{flower automaton}
$\Flower(T)$, which accepts $T^*$, to be the automaton over $S^\pm$ formed 
by attaching disjoint 
directed  loops, each labelled by an element $t \in T$
to a vertex $s_0$ that is the 
both the start state and the single accepting state.
\end{definition}
\mcomment{SR: added def of $\Flower(T)$. I think it is supposed to accept $T^*$ and not  $(T^\pm)^*$. But I'd like to clarify this. AD: I think it's correct now.}

As a consequence of \cref{thm:algorithm_terminates}
 and \cref{thm:weaklyLgood_implies_Lgood_submonoids}, we have the following result.
\begin{thm}\label{thm:submonoid_corollary}
	Let $(G, L)$ be a rational structure for a finitely presented group 
$G = \langle S|R\rangle $, let $T\subseteq (S^\pm)^*$ be a finite set of 
words and consider the regular language $Q_0:=T^*$ accepted by the flower automaton 
$\A_0:=\Flower(T^*)$ (\cref{def:flower}).
	If the submonoid $M = \mupi(Q_0)$ is weakly-$L$-\good\  with respect to $Q_0$, with constant $k$ (see \cref{def:WeaklyGoodQ}) then there exists $n\in \mathbb{N}$ such that the full set of $L$-representatives of $M$ is contained in  $Q_n$, the language accepted by the automaton $\A_n$ whose construction is described above.
\end{thm}

%

\section{Application to submonoids of automatic groups}\label{sec:autsubmonoid} 

In this section we use the argument of   Kharlampovich, Miasnikov and Weil \cite{kmw17} to turn the procedure 
of the previous section into a partial algorithm 
which will construct an automaton accepting the set of $L$-representatives of a given finitely generated monoid of an automatic group.
The main ingredient is a test to show whether or not the process has completed. The test works for submonoids because the automaton corresponding to a finite
generating set of a monoid has start state equal to its unique accept state. As we shall see, this means that the following lemma, from \cite{kmw17}, gives rise to the required test.  Throughout the section, the conventions for symmetric presentations of a group remain as in \cref{sec:Construction}.

For the remainder of this section let $G$ be an automatic group with generating set $S$ and $L\subseteq (S^\pm)^*$ be the language of this automatic structure.
\begin{lemma}[Lemma 4.11 in \cite{kmw17}]\label{lemma:kmw_technical}
	Given $h\in F(S)$ and a regular language $Q$ contained in $L$,
        there exists an algorithm that constructs an automaton accepting the set $\M_h(Q)=L\cap \mupi^{-1}(\mupi(Qh))$
        of all the $L$-representatives of the elements of $\mupi(Qh)$.
 \end{lemma}

 \begin{thm}\label{thm:subpgalg}
   Let $T$ be a finite subset of $S^\pm$,
   such that the submonoid $M=\mupi(T^*)$ of $G$ generated by $T$ is
   weakly-$L$-\good. Let $\A_0:= \Flower(T)$  (\cref{def:flower}),
and
   let $\A_0,\A_1,\ldots $ be the sequence of automata
   constructed in \cref{sec:Construction}. Then, for $i \ge 0$, the automaton
   $\A_i$ may be effectively constructed and
   there is an algorithm 
   which, on input $T$, 
   outputs an integer $n\ge 0$ 
   such that
     the language $L(\A_n)$ contains all $L$-representatives of  $M$.
     \end{thm}
 \begin{proof}
   First we describe the algorithm and then we verify that it works as required.
   The algorithm begins by finding an element $w_\ve\in L$ such that
   $\mupi(w_\ve)=1_G$. As $G$ is automatic this can be done using the multiplier automaton for $\ve$. After this pre-processing step, 
   for each $i\ge 0$ until it halts, the algorithm consists of $3$ steps. The input at
   each iteration is the automaton $\A_i$ and we set $Q_i:=L(A_i)$.
   In the first step the algorithm checks to see that $w_\ve$ is in $Q_i$.
   If not then the algorithm passes to the third step described below.
   If $w_\ve \in Q_i$ the algorithm carries out its second step in which it
   \begin{itemize}
   \item constructs an automaton for the language $Q_i':=L\cap Q_i$;
   \item constructs an automaton accepting $\M_t(Q_i')$, for each $t\in \{\ve\}\cup T$;
   \item tests to see whether or not $\M_t(Q_i')$ is contained in $Q_i'$.
   \end{itemize}
   If the answer to the test is ``yes'', for all $t\in\{\ve\}\cup T$,  the algorithm halts and outputs
   $i$ and $\A_i$. 
   Otherwise the algorithm continues to step $3$ in which \cref{construction:completion} is used to construct $\A_{i+1}$; and then  moves on to         iteration $i+1$.

   From \cref{thm:submonoid_corollary}, \cref{lemma:kmw_technical} and the standard theory of finite state automata,
   it follows that this procedure may be executed, and effectively constructs $\A_{i}$ at the $i$th iteration. To complete the proof  we show that
   it halts and that, when it does so, the final claim holds.

   First suppose that, for some $n\ge 0$, the language $Q_n:=L(\A_n)$ contains $L\cap \mupi^{-1}(M)$. Then, in particular,
   $w_{\ve}\in Q_n$ and $L\cap \mupi^{-1}(M)\subset L\cap Q_n=Q_n'$. As
   $\mupi(Q_n')=\mupi(Q_n)=M$ and $\mupi(t) \in M$, for all $t\in \{\ve\}\cup T$, we have $\mupi(Q_n't)\subseteq M$ so $L\cap \mupi^{-1}(\mupi(Q_n't))\subseteq L\cap \mupi^{-1}(M)
   \subseteq Q_n$. Hence $\M_t(Q_n')\subseteq Q_n$, for all $t\in \{\ve\}\cup T$. 

   Conversely suppose that, for some $n\ge 0$, $w_\ve \in Q_n$ and, for all $t\in \{\ve\}\cup T$,  we have $\M_t(Q_n')\subseteq Q_n$.
   Suppose $w\in Q_0$, so
   $w=t_1\cdots t_m$, where $t_i\in T$, $m\ge 0$. If $m=0$ then $w=\ve$. As $w_\ve\in Q_n'$, we have $Q_n'\neq \emptyset$, so all
   $L$-representatives of $\mupi(\ve)=\mupi(w_\ve)$ belong to $\M_\ve(Q_n')\subseteq Q_n$. 

   Now assume that $m\ge 1$, and that if $u$ is a product of fewer than $m$ elements of $T$, all $L$-representatives of $\mupi(u)$ belong to $Q_n$.
   Then the $L$-representatives of $\mupi(t_1\cdots t_{m-1})$ (of which there is at least one) belong to $Q_n'$, so the $L$-representatives of $\mupi(w)$ all
   belong to $\M_{t_n}(Q_n')\subseteq Q_n$; whence all $L$-representatives
   of elements of $\mupi(Q_n)=M$  belong to $Q_n$.

   We have shown that $L\cap \mupi^{-1}(M)\subseteq Q_n$ if and only if both $w_\ve\in Q_n$
   and $\M_{t}(Q_n')\subseteq Q_n$, for all
   $t\in \{\ve\}\cup T$. Therefore the algorithm halts with the required output.
 \end{proof}
 \begin{cor}
   In \mcomment{AD: added corollary about membership problem} the notation and terminology of \cref{thm:subpgalg} an automaton recognising
   $L\cap \mu^{-1}(M)$ may be effectively computed and  the submonoid membership
   problem of $M$ is constructively decidable. 
 \end{cor}
 \begin{proof}
   To obtain an automaton accepting exactly $L\cap \mupi^{-1}(M)$ we use a standard algorithm for
   an automaton for the language $L\cap L(\A_n)$, where $\A_n$ is output from the algorithm of \cref{thm:subpgalg}. For the membership problem we use the argument of \cref{prop:AutomaticLRecognisableMembership} and the automaton $A_n$.
 \end{proof}
 \begin{remark}
   In general, \mcomment{AD: added a comment to emphasise lack of constructibility} although $L$-\good\ rational subsets of a group $G$ have decidable membership problem,  whether this is constructive decidability is open.
   For example, 
   the proof of \cref{prop:AutomaticLRecognisableMembership}  breaks down if $Q_0$ is an arbitrary regular subset of $(S^\pm)^*$. The difficulty is to find an analogue of the generating set $T$ and
   then an analogue of the statement that given that $w_{\ve}\in Q_n$ then the fact that $\M_t(Q'_n)\subseteq Q_n$, for all $t$, implies that
   $L\cap \mupi^{-1}(M)$ is contained in $Q_n$. Both of these difficulties arise when $Q_0$ has distinct start and final state. In the first author's thesis
   further analysis of the general case  of $L$-\good\ rational subsets is carried out and some (technical) conditions are identified, under which a halting
   test can be devised. 
   \end{remark}

\section{Application to submonoids of surface groups} 

 The goal of this section is  to give examples of weakly-$L$-\good\ submonoids that lie inside surface groups, where $L$ is the language of geodesics relative to a standard generating set. We use tools from small cancellation theory to ensure that the necessary convexity conditions hold. \mcomment{JB: specified $L$}

 \begin{defn}
 	The \emph{surface group} $\Sigma_g$ of genus $g\geq 1$ is the group with presentation
 	$$\Sigma_g \equiv \langle a_1, b_1 \ldots a_g, b_g \mid [a_1, b_1]\cdots[a_g, b_g] = 1\rangle.$$
 \end{defn}
 We shall write $Y_g$ for the set of generators in the presentation above, and $r_g$ for the single relator.

 The group $\Sigma_g$ is the fundamental group of a closed, compact, orientable surface $S_g$ of genus $g$. Except where otherwise stated, we assume that ${g>1}$ throughout this section. In this case $\Sigma_g$ is a hyperbolic group and therefore automatic, with a rational structure $L$ given by geodesics with respect to $Y_g$. \mcomment{JB: clarified that geodesics are with respect to the given generators}
 
 Given a set $R$ of relators over a generating set $S$, the \emph{symmetrisation} $R_*$ is the set of cyclic permutations of words in $R$, together with their inverses. We write $R_g$ for the symmetrisation of $\{r_g\}$. Recall that Dehn's rewriting system for $\Sigma_g$ has the following rewriting rules:
 	\begin{enumerate}
 		\item $ss^{-1} \to \varepsilon$, for a letter $s$ in $Y_g^\pm$.
 		\item $u\to v$, for words $u, v$ in $(Y_g^\pm)^*$ such that $uv^{-1}\in R_g$ and $|u|>|v|$. 
 	\end{enumerate}
 	A theorem of Dehn states that any word in ($Y_g^\pm)^*$ which represents the identity of $\Sigma_g$ can be reduced to $\varepsilon$ by repeated applications of these rewriting rules (Dehn's algorithm).

        It is useful to think of the Cayley graph $\Gamma(\Sigma_g, Y_g)$ as a tessellation of the hyperbolic plane $\mathbb{H}^2$ by $4g$-gons with edges labelled by the relator $r$. The second of Dehn's rewriting rules replaces any path going around more than half of a $4g$-gon with the shorter route using the complementary edges. See \cref{fig:dehn_reduced} for some genus $2$ examples.
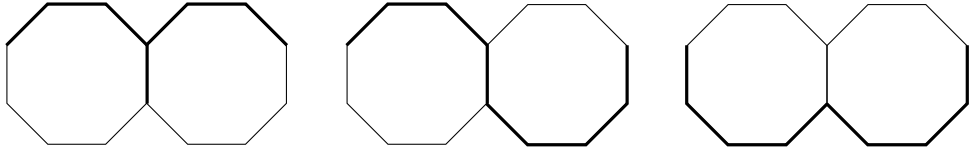
\begin{figure}[ht]
 	\centering
\begin{tikzpicture}
  \foreach \x in {0,1,2}{
    \begin{scope}[xshift=\x*4.5cm]
          \node[name=s\x-1, regular polygon, regular polygon sides=8, inner sep=0, minimum size = 2cm, draw] at (0,0) {};
          \begin{scope}[xshift=1.85cm]
            \node[name=s\x-2, regular polygon, regular polygon sides=8, inner sep=0, minimum size = 2cm ,draw] at (0,0) {};
          \end{scope}
    \end{scope}
  }
\draw[very thick]  (s0-1.corner 3) -- (s0-1.corner 2) --(s0-1.corner 1) -- (s0-1.corner 8) -- (s0-1.corner 7);
            \draw[very thick]  (s0-2.corner 3) -- (s0-2.corner 2) --(s0-2.corner 1) -- (s0-2.corner 8);
 \draw[very thick]  (s1-1.corner 3) -- (s1-1.corner 2) --(s1-1.corner 1) -- (s1-1.corner 8) -- (s1-1.corner 7);
 \draw[very thick]  (s1-2.corner 4) -- (s1-2.corner 5) --(s1-2.corner 6) -- (s1-2.corner 7) -- (s1-2.corner 8);
 \draw[very thick]  (s2-1.corner 3) -- (s2-1.corner 4) --(s2-1.corner 5) -- (s2-1.corner 6) -- (s2-1.corner 7);
            \draw[very thick]  (s2-2.corner 4) -- (s2-2.corner 5) --(s2-2.corner 6) -- (s2-2.corner 7) -- (s2-2.corner 8);  
        \end{tikzpicture}
 	\caption{Three examples of words in $\Sigma_2$ depicted as paths (in bold) along the Cayley graph of $\Sigma_2$, which tessellates the hyperbolic plane with octagons. The leftmost path is not Dehn-reduced because it exhibits backtracking, the middle one is not Dehn-reduced because it goes around $5$ edges of an octagon. The rightmost path represents a Dehn-reduced word.}
 	\label{fig:dehn_reduced}
 \end{figure}

\subsection{Results from small cancellation theory}

We shall need to introduce a result, Theorem 9.4, from \cite{mccammondwise02}. We recall some standard definitions from small cancellation theory, which can be found in \cite{ls79}, or in the small cancellation theory appendix to \cite{gh90}.

\begin{defn}
Let $P=\langle S\mid R\rangle$ be a group presentation. A word $u$ over $S$ is a \emph{piece} with respect to $P$ if $u$ is a common prefix of two distinct relators in $R_*$.
\end{defn}
\noindent For the standard presentation of $\Sigma_g$, the pieces are precisely the $1$-letter words.

\begin{defn}
Let $P=\langle S\mid R\rangle$ be a group presentation.
\begin{itemize}
\item For a natural number $p$, we say that a $P$ satisfies condition $\Ccond(p)$ if no element of  $R_*$ is the concatenation
  of fewer than  $p$ pieces. \mcomment{AD: changed the definition of  $\Ccond(p)$. }
\item For $q>2$ we say that $P$ satisfies condition $\Tcond(q)$ if for every $\ell\in\{3, 4, \ldots, q-1 \}$, and for every sequence $\{r_1, r_2, \ldots r_\ell\}$ in $R_*$ such that $r_i\neq r_{i+1}^{-1}$ for all $i$, and such that $r_\ell\neq r_1^{-1}$, at least one of the products $r_1r_2,\ldots, r_{\ell-1}r_\ell, r_1r_{l-1}$ is freely reduced. 
\end{itemize}
\end{defn}	
A presentation which satisfies both of the conditions $\Ccond(p)$ and $\Tcond(q)$ is said to be $\Ccond(p)$--$\Tcond(q)$.  We note that the standard presentation for $\Sigma_g$ is $\Ccond(4g)$--$\Tcond(4g)$; in particular, it follows that this presentation is $\Ccond(4)$--$\Tcond(4)$.
	
Recall that a Van Kampen diagram for the presentation $\langle S\mid R\rangle$
of a group $G$ is a finite, connected, contractible \mcomment{AD: added contractible} cell complex, embedded into $\R^2$, such that the edges are directed and labelled by elements of $S$, and such that the labels along the boundary of each $2$-cell spell a element of $R_*$.

A \emph{boundary cycle} of a Van Kampen diagram $X$ is a closed path $p$, based at a vertex of the boundary $\partial X$ and winding exactly
once around $X$ (with one of the two possible orientations). The label of a boundary path is a \emph{boundary word} of $X$.
 
A pair of discs $D_1$ and $D_2$ of a Van Kampen diagram is said to be
\emph{cancellable} if the boundaries of $D_1$ and $D_2$ contain a directed
edge $e$ such that the labels of $D_1$ and $D_2$ read by starting with the
edge $e$ are equal. A Van Kampen diagram which
contains no pair of cancellable discs it is said to be \emph{reduced}.
It is well known (see for example \cite{mccammondwise02}) that a word $w\in (S^\pm)^*$  represents the identity
element of $G$ if and only if it is a boundary word of a Van Kampen diagram for  $\langle S\mid R\rangle$.\mcomment{AD: added reduced diagram.}

An \emph{arc} in the Van Kampen diagram $X$ is a simple path in the $1$-skeleton of $X$ \mcomment{AD: added $1$-skeleton}  whose internal vertices have degree $2$; an arc is \emph{maximal} if it is not properly contained in any other arc. An arc is an \emph{internal arc} if its interior lies in the interior of $X$, and a \emph{boundary arc} if it lies entirely on the boundary of $X$. 
We notice that the labels along a maximal internal  arc in a reduced
\mcomment{AD: added reduced} Van Kampen diagram spell a maximal piece.
  Hence for the usual presentation of $\Sigma_g$, all  maximal internal arcs have length $1$. (That is to say, there are no internal vertices with degree $2$.)
  The following definitions are taken from \cite{mccammondwise02}.
\begin{defn}\label{def:spur} 
A \emph{spur} is a $1$-cell which is not in the boundary of a $2$-cell and which is attached to the rest of the diagram at only one end. An $i$-\emph{shell} is a $2$-cell with exactly $i$ 
maximal arcs of its boundary lying in the interior of the diagram and exactly
one maximal boundary arc.\mcomment{AD: added 1 bdry arc} 
\end{defn}	

For the next definition we follow \cite{wise04}.
\begin{defn}\label{def:ladders}	
	A \emph{ladder} is a Van Kampen diagram $X$ which is  the union of a sequence of closed 1-cells and 2-cells $c_1,\ldots, c_n$, such that for $1 < j < n$,
	there are exactly two components in $X-c_j$, such that
	$X-c_1$ and $X-c_n$ each has exactly one component, and such that any $c_i$ which is a 1-cell is not contained in any other closed $c_j$.
\end{defn}\mcomment{JB: changed ladder notation from $L$ to $X$  since $L$ is a language.} 
A ladder can be considered as a diagram with two distinguished points $x$ and $y$ on its boundary, such that if $P$ and $Q$ are the two boundary paths from $x$ to $y$, then every $2$-cell has at least one edge from $P$ and another from $Q$, and such that every edge not contained in a $2$-cell is common to both $P$ and $Q$. (We consider the paths $P$ and $Q$ as the rails of the ladder.) 

The following result is \cite[Theorem 9.4]{mccammondwise02}. 
\begin{prop}
	If $X$ is a Van Kampen diagram for a $\Ccond(4)$--$\Tcond(4)$ presentation, then one of the following holds:
	\begin{enumerate}
		\item $X$ contains at least $3$ spurs or $i$-shells with $i\leq 2$;
		\item $X$ is a ladder;
		\item $X$ consists of a single vertex or a single $2$-cell.
	\end{enumerate}
\end{prop}

\subsection{Dehn-reduced words and geodesics in surface groups}

Note that an arbitrary word $w$ of $(Y_g^\pm)^*$ is not necessarily reduced to a geodesic word by Dehn's algorithm. The rightmost path in \cref{fig:dehn_reduced} represents a Dehn-reduced word which is not geodesic. (Complete rewriting systems for surface groups are known: see for instance \cite{hermiler94} or \cite{gn97}.) We shall show, using small cancellation properties of our presentation for $\Sigma_g$, that Dehn-reduced words asynchronously fellow-travel with geodesics. This gives us a tool to find weakly-$L$-\good\ languages which describe submonoids of surface groups, and hence a class of $L$-recognisable submonoids.

We need the following lemma.
 \begin{lemma}\label{lemma:dehn_red_no_shells}
Let $X$ be a reduced \mcomment{AD:added reduced} Van Kampen diagram for the surface group $\Sigma_g$. Suppose that a boundary cycle of $X$ starting at a given vertex is labelled by a concatenation of Dehn-reduced words $w_1,\ldots, w_n$. Suppose also that $X$ has an $i$-shell $Y$ with $i\leq 2g-1$. Then the maximal arc which is the intersection of $Y$ with the boundary of $X$, contains edges from two consecutive words $w_jw_{j+1}$ or $w_{n}w_1$. 
 \end{lemma}
 \begin{proof}
 Since the only relator for $\Sigma_g$ has length $4g$, it follows that the $2$-cell $Y$ has $4g$ edges.
Since all internal edges are maximal arcs, at most $2g-1$ of these edges are internal so we see that $Y$ has at least $2g+1$ boundary edges, and that the labels for these edges spell a subword of an element of $R_g$. 	But since each $w_i$ is Dehn-reduced, none can contain a subword of an element of $R_g$ of length $2g+1$, and so the boundary edges of $Y$ must contain labels from two consecutive words $w_iw_{i+1}$. 
 \end{proof}

We can now obtain the following result.
\begin{lemma}\label{lemma:Dehn_geodesic_fellow_travel}
	Let $u$ be a Dehn-reduced word in $(Y_g^\pm)^*$,  and let $v$ be a geodesic representing the same element of $\Sigma_g$. (So $v$ is also Dehn-reduced.) A  reduced \mcomment{AD:added reduced} Van Kampen diagram for the word $uv^{-1}$ must be a single vertex, a single $2$-cell, or a ladder. In particular, $u$ asynchronously $(g+1)$-fellow travels with $v$ (with upper bound of $2g$ on asynchronicity). \mcomment{AD: changed to  $g+1$-fellow travels.} 
\end{lemma}
\mcomment{AD: The proof appears to show that $u$ asynchronously $g+1$-fellow travels with linear bound on the asynchronicity of $2g$ (If I counted correctly.)
  Doesn't that give $L$-\good\ with constants $g+1$ and $2g$?} 
\begin{proof}
We see from \cite[Corollary 2.8]{wise04} that if $X$ is a diagram for a $\Ccond(4)$--$\Tcond(4)$ presentation, and
if $uv^{-1}$ is the boundary cycle of $X$, then either $u$ or $v$ contains all of the boundary edges of a $1$-shell or a $2$-shell of $X$, or else $X$ is a single vertex, a $2$-cell, or a ladder.
Let $X$ be a Van Kampen diagram for $uv^{-1}$, and let $u_0$ and $v_0$ be the vertices corresponding to the start of the words $u$ and $v^{-1}$ respectively. 
By Lemma \ref{lemma:dehn_red_no_shells} any $1$- or $2$-shell of $X$ must include edges from both $u$ and $v^{-1}$, and so neither $u$ nor $v^{-1}$ contains all of the boundary edges for for any $1$- or $2$-shell. So $X$ is a vertex, a $2$-cell, or a ladder.	
	
\mcomment{AD: expanded this paragraph.}
If $X$ is a vertex then clearly $u$ and $v$ fellow-travel. If $X$ is a
$2$-cell then, as $u$ and $v$ are Dehn-reduced, both have length $2g$ and
so the result holds. \mcomment{i.e. $L$-proximacy, with constants $g$ and $2g$}
Assume then that $X$ is a ladder. In this case its endpoints are $u_0$ and $v_0$, and the rails are $u$ and $v$. Each cell $c_i$ forming the ladder must intersect both rails, and this observation, together with the fact that 
cells along the Dehn-reduced word $u$ intersect the boundary of $X$ in at most  $2g$ edges,
and internal arcs have length at most $1$,  shows that every vertex of $u$ is at distance less than $g+1$ from a vertex of $v$ and that these vertices
of $v$ are separated by at most $2g-2$ edges of $v$. \mcomment{$L$-proximate with
  constants $g+1$ and $2g-2$.}
\end{proof}

\subsection{Constructing $L$-recognisable submonoids of surface groups}

We \mcomment{AD: added ``constructively'' to decidability statements}start with an immediate corollary to Lemma~\ref{lemma:Dehn_geodesic_fellow_travel}.
\mcomment{LAM Maybe one of these surrounding results could be a theorem? They are not super strong but they provide the main source of examples, and JB pointed out we barely have theorems. I am not sure which ones could be.}
\begin{lemma}\label{lemma:DehnReduced}
	A Dehn-reduced language $Q$ over the generators of $\Sigma_g$ is weakly-$L$-\good\ in the surface group $\Sigma_g$. In particular, if $T$ is a generating set for a submonoid $M\subset\Sigma_g$, and if every word in $T^*$ is Dehn-reduced, then $M$ is $L$-recognisable, and $M$ has constructively decidable membership problem.
\end{lemma}

Before we give examples of submonoids to which this result applies, we note an immediate way to strengthen it.
\begin{definition}
  Let $G$ be a group with presentation $\langle S|R\rangle$, and let $u, v$ be words in $(S^\pm)^*$. We say that $u$ yields $v$ in a \emph{simultaneous Dehn-reduction}, $u\xrightarrow[\Dsim]{} v$, if we can write $u$ and $v$ as a concatenation of subwords $u = u_1\ldots u_t$ and $v = v_1\ldots v_t$ such that $v_i$ can be obtained from $u_i$ by applying a Dehn rewriting step. Given a sequence of
  words $u=d_0, \ldots ,d_N=v$ of length $N\ge 0$, where $d_{i-1} \xrightarrow[\Dsim]{} d_{i}$, for $1\le i\le N$, we say that $u$ is $N$ \emph{simultaneous Dehn rewriting steps} from $v$. 
\end{definition}

\begin{thm}\label{thm:Dehn_reduced_Q_constructively_decidable}
  Let $N$ be a non-negative integer and let \mcomment{LAM this raises a natural question to me that maybe we should say something about: do we have a f.g. submonoid in a surface group that is not $k$ steps of being Dehn reduced for some $k$?\\[1em]
    AD: do we know of a f.g. submonoid of a surface group which is not weakly-L-proximate? \\[1em]
    Do all words of an  $L$-proximate language  lie within $N$ simultaneous Dehn steps of a
    Dehn-reduced word, for some fixed $N$? Yes, is my guess, based on the fact that surface groups satisfy a linear isomperimetric inequality.}
  $Q$ be a language over $Y_g$ all of whose words are at most $N$ simultaneous Dehn rewriting steps away from a Dehn-reduced word. Then $Q$ is weakly-$L$-\good\ with constant $k=(2N+1)g+1$. If $Q= T^*$, then the submonoid $M=\mupi(T^*)$ is $L$-recognisable and has constructively
  decidable membership problem. \mcomment{JB: changed constant from $(g+1)(N+1)$ to $(2N+1)g +1$.}
\end{thm}
\begin{proof}
We observe that if $u$ yields $v$ in a simultaneous Dehn-reduction, then $u$ asychronously $2g$-fellow \mcomment{JB: I think $2g$ is right here, rather than $g+1$ as previously stated.} travels with $v$. Let $Q'$ be the language of Dehn-reduced words obtainable from words in $Q$ by at most $N$ simultaneous Dehn-reductions. Then for every word $u\in Q$ there is $v\in Q'$ such that $u$ asynchronously $2Ng$-fellow travels with $v$. Since $v$ is Dehn-reduced, there is a geodesic $w$ representing the same element of $G$ such that $v$ asynchronously $(g+1)$-fellow travels with $v$ by \cref{lemma:Dehn_geodesic_fellow_travel}. So $u$ asynchronously $k$-fellow travels with $w$, where $k=(2N+1)g+1$, and so $Q$ is weakly-$L$-\good\ with constant $k$.
\mcomment{JB: added proof. Do we state the connection between $L$-recognisability and decidable membership problem somewhere? I can't find it.}

If $Q$ is a submonoid $T^*$ of $(S^\pm)^*$ then $Q$, is $L$-recognisable by \cref{proposition:SubmonoidGoodRecognisable}. 
\end{proof}


Now we see how to build a family of Dehn-reduced submonoids. \mcomment{JB: family of submonoids}
\begin{prop}\label{prop:dehn_reduced_construction}
	Let $T\subset (S^\pm)^*$ be such that
	\begin{enumerate}
		\item for all $t\in T$, $t$ is Dehn-reduced and has length at least $g$;
		\item for all $t\in T$, neither its length $g$ prefix nor its length $g$ suffix are a prefix of any word of $R_g$;
                \item If a letter is an initial (final) letter of a word in $T$ then its inverse is not a final (initial) letter of a word in $T$.\mcomment{AD: changed the last item.}
	\end{enumerate}
	Then the language $T^*$ is Dehn-reduced. Therefore $T^*$ is weakly-$L$-\good, and the submonoid it generates is $L$-recognisable.
\end{prop} 
This provides a simple way of constructing Dehn-reduced submonoids of a surface group.
\mcomment{AD: removed figure}
\begin{example}\label{ex:dehn_reduced_monoid}
	Let $\Sigma_2=\langle Y_2 \;|\; r_2\rangle$ be the standard presentation for a genus 2 surface group, with $Y_g=\{a, b, c, d\}$ and $r_2= aba^{-1}b^{-1}cdc^{-1}d^{-1}$. Let $T=\{w_1, w_2, w_3, w_4\}$, where
	\begin{alignat*}{3}
		w_1 = &\; a^{-1}c       && u_1   &&cc,\\
		w_2 = &\; bc       && u_2 \;&&ac^{-1}, \\
		w_3 = &\; bd^{-1}  && u_3   && ad, \\
		w_4 = &\; b^{-1}d\;&& u_4   && c^{-1}d,		
	\end{alignat*}  
\mcomment{LAM Ive aligned these words to visualise the separation between the first $g$ letters, the last $g$ letters and the infix. If this is confusing, happy to have this desaligned.} and the words $u_1, \ldots u_4$ can be arbitrarily chosen as long as the resulting words $w_1, \ldots, w_4$ remain Dehn-reduced. Then, the set $T$ satisfies the hypothesis of \cref{prop:dehn_reduced_construction} and $T^*$ is Dehn-reduced. For a concrete example it suffices to take, among many possibilities, 
\begin{alignat*}{3}
	u_1 = &\; d, && \text{ so that } w_1 &&= a^{-1}cdc^{2},\\
	u_2 = &\; d^{-1}, && \text{ so that } w_2 &&= bcd^{-1}ac^{-1}, \\
	u_3 = &\; c^{-1}, && \text{ so that } w_3 &&= bd^{-1}c^{-1}ad, \\
	u_4 = &\; cd^{-1}, && \text{ so that } w_4 &&= b^{-1}d   cd^{-1}  c^{-1}d.		
\end{alignat*}
\mcomment{AD: To ensure that neither Theorem 4.1 or Theorem 4.5 of \cite{fg25} applies $w_1$ has been changed from $ac u_1cc$ to $a^{-1}c u_1 cc$.}
\end{example}
In recent work by Foniqi and Gray \cite{fg25}, the authors provide some results regarding the membership problem of certain submonoids of the surface group. \mcomment{LAM I dont understand well their graded case in section 6, it would be helpful if someone could take a look at it and guarantee that their section 6 does not cover our results} More specifically, they guarantee decidability of the submonoid membership problem in a surface for Magnus submonoids \cite[Theorem 4.4]{fg25}, for submonoids where the $y$-exponent of all the words in $T$ is controlled for some $y\in Y_g$ \cite[Theorem 4.1]{fg25}, and for certain graded submonoids \cite[\S 6]{fg25}. We note that the submonoids of our examples above are not covered
by \cite[Theorem 4.1]{fg25} or  \cite[Theorem 4.4]{fg25} and not obviously within the scope of results of \cite[\S 6.3]{fg25}.

\mcomment{AD: are there other properties that can be derived from the
  folding of these monoids?}

\blockcomment{

\section{Application to rational subsets of right-angled Artin groups} 

In this section we shall show that in the case of rational subsets of right-angled Artin groups, given an appropriate rational structure
$L$ we may construct an algorithm from the procedure described in \cref{subsec:folding}. 

~\\[1em] Definition of raag. Description of geodesics in raags. $\supp(w)$ for an element of $(S^\pm)^*$ and $\lk(s)$ for $s\in S^\pm$.
If $w$ is a word in $(S^\pm)^*$ then we say that a word $w'$ is obtained from $w$ by a \emph{shuffle reduction}
if $w=w_0aw_1a^{-1}w_2$ and $w'=w_0w_1w_2$, where $a\in S^\pm$, $w_i\in (S^\pm)^*$, for $i=0,2$ and  $w_1\in (\lk(a))^*$.
In this case we write $w\xrightarrow[\SR]{}w'$. 
From \cite{somethingorother} a word $w\in (S^\pm)^*$
is a geodesic representative of $\mupi(w)$ if and only if its length cannot be reduced by a shuffle reduction.
A shuffle reduction, with the description above,  is said to be \emph{minimal} if $w_1$ is geodesic. As $w_1$ may be the empty word
 a free reduction is a special case of a minimal shuffle reduction. If 
$w\in (S^\pm)^*$ then a geodesic representative $w'$ of $\mupi(w)$ may be obtained by performing a finite sequence of minimal
shuffle reductions to $w$; that is,  there is a sequence $w=w_0, \ldots, w_n=w'$, such that
$w_i\xrightarrow[\SR]{}w_{i+1}$ is a minimal shuffle reduction, for $0\le i\le n-1$ and $w'$ is geodesic (see \cite{thesameagain}).
Moreover, the number of (minimal) shuffle reductions required to rewrite $w$ to a geodesic depends only on $w$ and we denote
this number by $\SR(w)$.  A word $w=aw_1a^{-1}$ where, as above, $a\in S^\pm$, $w_1\in (\lk(a))^*$ and $w_1$ is geodesic is called a
\emph{minimal shuffle reduction word}, and we call its expression as $awa^{-1}$ its \emph{standard factorisation}.\\[1em]

Throughout this section let $G$ be a right-angled Artin group with presentation $\langle S\,|\, R\rangle$,  where $R$ 
is closed under taking inverses and cyclic permutation of words, let  $L$ be the regular language
consisting of all geodesics for $G$ and let $\A$ be
an automaton over $S^\pm$ such that the language $Q$ of $\A$ is $L$-\good\ over $G$. We do not assume that we know the
corresponding constants, $k$ and $c$, of $L$-{\goodness}. The procedure of \cref{subsec:folding} produces a sequence $\A_0=\A,\A_1,\A_2,\ldots$
of automata such that, for some $i\ge 0$, the set of all $L$-representatives of $K=\mu(Q)$ is contained in $L(A_i)$. 
To make an algorithm of this procedure requires a test to decide whether or not $\A_i$ has the latter property. Our test has two parts, which we call Test 1 and Test 2.
\mcomment{SR: Tidied up naming of the tests, here and below}
\begin{enumerate}
\item[\textbf{Test 1.}] Check that $L(\A_i)$ contains at least one $L$-representative of $g$, for all $g\in K$.
\item[\textbf{Test 2.}]
  If Test 1 is successful check, for each word $w$ in $L\cap L(\A_i)$, that every element $w'$ of $L$ such that
  $\mu(w')=\mu(w)$ belongs to $L(\A_i)$.
\end{enumerate}  
Clearly  $L(\A_i)$ contains all representatives of $K$ if and only if both parts of the test are successful.

To simplify notation assume $\A_i$ 
 is the (deterministic, trim) finite state automaton
$(S^\pm,V,E,V_F,v_0)$. As in \cref{sec:Notation}, we may regard $\A_i$ as a labelled, directed graph $(V,E)$ with labelling function $l$  and we
shall need some notation for such graphs.
We may refer to the edge $(u,s,v)\in E$ simply as
$(u,v)$ when we do not need to consider its label $s$.
For $u \in V$ define the set $\out(u)=\{v\in V\,|\, (u,v)\in E\}$. 
As mentioned on page \pageref{page:path} we may identify a path with its image
in $\A_i$, which we write as a sequence
$u_0,e_1,u_1,\ldots , e_n,u_n$, where $u_i\in V$, $e_i\in E$ and $e_i=(u_{i-1},u_i)$, and as before such a path has length $n$.
We shall sometimes define a path by writing its edge sequence, in the form $p=e_1e_2\cdots e_n$, where $p$ is  the path above. 
 A path with \emph{initial}
 vertex $u_0:=x$ and \emph{terminal} vertex $u_n:=y$ is called an $(x,y)$\emph{-path}. If $Y$ is a subset of $V$ then an $(x,y)$-path such
 that $y\in Y$ is called an $(x,Y)$-path. 
 In this terminology $L(\A_i)$ is the set of words $w$ such that
$w$ is the label of a $(v_0,V_F)$-path in $\A_i$. 

A path is called \emph{reduced} if its label is reduced as a word in $F(S)$ and  \emph{geodesic} if its label is geodesic with
respect to the presentation $\langle S\,|\,R\rangle$.  A path is called a \emph{minimal shuffle reduction} path 
if its label is a
minimal shuffle
reduction word.  Moreover, in this case if $l(p)$ is the minimal shuffle reduction word with standard factorisation $awa^{-1}$
then the factorisation of $p$ 
as $e_0p_1e_{n+1}$, where $n=|w|$, $e_i\in E$, $l(e_0)=a$, $l(e_{n+1})=a^{-1}$ and $l(p_1)=w$ is called the \emph{standard factorisation} of $p$.

\subsection*{Test 1}\label{subsec:part1}
\mcomment{SR: reworded}
 Note that we may assume that $i\ge 1$ so that $Q_i := L(\A_i)$ contains all reduced words obtained by free reduction from its elements. 
To simplify notation set $\B:=\A_i=(S^\pm,V,E,V_F,v_0)$
in the description of Test 1 that now follows.
\subsubsection*{Test 1.} \label{sec:test1}
Let $p$ be an acyclic minimal shuffle reduction $(u,v)$-path which has standard factorisation $e_0p_1e_{n+1}$, where $n\ge 0$,
$e_0=(u,a,u_0)$, $e_{n+1}=(u_n,a^{-1},v)$ and $p_1$ is a $(u_0,u_n)$-path with label $l(p_1)=w$.
We wish to check that, for some vertex $v'$,  $\B$ has a $(u,v')$-path $q_w$ with label $w$ and an edge $e'_{n+1}=(v',a,u_n)$.  If this is true then,
as $\B$ is the result of an application of Benois' theorem and 
the path $e'_{n+1}e_{n+1}$ has label
$aa^{-1}$, it follows that for each edge $(v,s,v'')$ of $\B$ there is a corresponding edge $(v',s,v'')$.  

We shall perform our verification
for all geodesic $(u_0,u_n)$-paths  with labels in $\lk(a)^*$. We note that the set
of words which can appear as labels of such paths is regular, being the set  of geodesic
elements accepted by the automaton $\X = (lk(a), V, E_a, \{u_n\}, u_0)$, where $E_a$ is the set of edges of $\B$ with labels in $\lk(a)$.  
For our test, we want to show that for every geodesic word $z$ of $L(\X)$ there is a $(u,U_a)$-path in $\B$ with label $z$, where 
$U_a$ is the set of vertices $\{x\in V\,|\, \exists (x,a,u_n)\in E\}$. 
But the set of geodesic $(u,U_a)$-paths is also regular, being the language of geodesic words accepted by 
the automaton $\Y= (\lk(a), V, E_a, U_a, u)$.
It will be enough, then, to check that $L\cap L(\X)$ is contained in $L\cap L(\Y)$.  There are standard methods for constructing the intersection
or regular langauges and 
for testing whether one regular subset of $(S^\pm)^*$ is contained in
another (ref Hopcroft \& Ullman), so this check can be performed.

Our test, then, proceeds as follows. For each $a \in S^\pm$ and all $u,v,u_0,u_n\in V$, if there exist edges $e_l=(u,a,u_0)$ and $e_r=(u_n,a^{-1},v)$
then 
 check that $L\cap L(\X) \subseteq L\cap L(\Y)$, where $\X$ and $\Y$ are as above. If the test is
passed for all $a$ and all such pairs of edges $e_l$ and $e_r$ then the verification is complete and $\B$ passes Test 1.

\ajd{Test 1 is too strong. The original idea was to test that all geodesics obtained from words in $L(\A_0)$ by sequences of minimal shuffle reductions
  belong to $\B$. For that the test should only be on  minimal shuffle reduction paths which are involved in such sequences of reductions. As it
  is it's applied to all minimal shuffle reduction paths.}

To prove that this test succeeds if and only if $\B$ contains at least one $L$-representative of every element of $K$ we need to demonstrate that (i) 
it fails if there is an element $g\in K$ such that $\mu^{-1}(g)\cap L(\B)$ contains no element of $L$ and  (ii)  it succeeds
if $\mu^{-1}(g)\cap L(\B)\cap L\neq \emptyset$, for all $g\in K$.

\ajd{got to here}

\subsection*{Test 2}\label{subsec:part2} Check there are no missing squares (after intersecting with $L$).

}

\bibliographystyle{amsplain} 
\bibliography{refs}

@incollection{bs21,
  author       = {Laurent Bartholdi and
                  Pedro V. Silva},
  editor       = {Jean{-}{\'{E}}ric Pin},
  title        = {Rational subsets of groups},
  booktitle    = {Handbook of Automata Theory},
  volume      = {2},
  pages        = {841--869},
  publisher    = {EMS Press, Z{\"{u}}rich,
                  Switzerland},
  year         = {2021},
  url          = {https://doi.org/10.4171/Automata-2/1},
  doi          = {10.4171/AUTOMATA-2/1},
  bibsource    = {dblp computer science bibliography, https://dblp.org}
}

@article{bkl22,
  title={Folding-like techniques for $\mathrm{CAT}(0)$ cube complexes}, 
  volume={173}, 
  DOI={10.1017/S0305004121000645}, 
  number={1}, 
  journal={Math.\ Proc.\ Cambridge Philos.\ Soc.\  }, 
  author={Ben–Zvi, Michael and Kropholler, Robert and Lyman, Rylee Alanza}, 
  year={2022}, 
  pages={227--238}}

@article{benois,
        title     ="Parties rationelle du groupe libre",
	author    ="Benois, Mich\`ele",
	journal   ="C. R. Acad.\ Sci.\ Paris, S\`{e}r.\ A",
	volume    =269,
	pages     ="1188--1190",
	year      ="1969"
}

@book{bridson11,
	title={Metric Spaces of Non-Positive Curvature},
	author={Bridson, Martin R. and Haefliger, André},
	isbn={9783540643241},
	lccn={99038163},
	series={Grundlehren der mathematischen Wissenschaften},
	url={https://books.google.co.uk/books?id=3DjaqB08AwAC},
	year={2011},
	publisher={Springer Berlin, Heidelberg}
}

@ARTICLE{dl21,
	title     = "Subgroups of right-angled {C}oxeter groups via {S}tallings-like
	techniques",
	author    = "Dani, Pallavi and Levcovitz, Ivan",
	journal   = "J. Comb.\ Algebra",
	publisher = "European Mathematical Society - EMS - Publishing House GmbH",
	volume    =  5,
	number    =  3,
	pages     = "237--295",
	month     =  oct,
	year      =  2021,
}

@book{echlpt92,
	title={Word Processing in Groups},
	author={Epstein, D. B. A. AND Cannon, J. W. AND Holt, D. F. AND Levy, S. V. F. AND Paterson, M. S. AND Thurston, W. P.},
	isbn={9781040162354},
	lccn={91046119},
	url={https://books.google.es/books?id=lH2NEQAAQBAJ},
	year={1992},
	publisher={CRC Press}
}

@misc{fg25,
	title         = "Magnus submonoids and membership problems in one-relator,
	surface and hyperbolic groups",
	author        = "Foniqi, Islam and Gray, Robert D.",
	note = {preprint at https://arxiv.org/abs/2412.04932},
	month         =  sep,
	year          =  2025,
	copyright     = "http://creativecommons.org/licenses/by/4.0/",
	archivePrefix = "arXiv",
	primaryClass  = "math.GR",
	eprint        = "2509.24480"
}

@book{gh90,
	editor    = {Ghys, E. AND de la Harpe, P. },
	title     = {Hyperbolic Groups},
	series    = {Progress in Mathematics},
	volume    = {111},
	publisher = {Birkhäuser Basel},
	year      = {1990},
	doi       = {10.1007/978-3-0348-8047-8},
	url       = {https://books.google.com/books/about/Hyperbolic_groups.html?id=FIGBAAAAIAAJ},
}

@article{gn97,
	author = {Rostislav Grigorchuk and Tatiana Nagnibeda},
	title = {Complete growth functions of hyperbolic groups},
	journal = {Invent. Math.},
	volume = {130},
	year = {1997}, 
	number = {1}, 
	pages = {159--188}
}

@article{gs91,
	ISSN = {0003486X, 19398980},
	URL = {http://www.jstor.org/stable/2944334},
	author = {S. M. Gersten and H. B. Short},
	journal = {Ann.\ of Math.\ },
	number = {1},
	pages = {125--158},
	publisher = {[Annals of Mathematics, Trustees of Princeton University on Behalf of the Annals of Mathematics, Mathematics Department, Princeton University]},
	title = {Rational Subgroups of Biautomatic Groups},
	urldate = {2025-03-14},
	volume = {134},
	year = {1991}
}

@article{hermiler94,
	author = {Susan M. Hermiler},
	title = {Rewriting systems for Coxeter groups},
	journal = {J. Pure Appl. Algebra},
	volume = {92},
	number = {2},
	year = {1994}, 
	pages = {137-–148}
}

@book{hrr17,
	title={Groups, Languages and Automata},
	author={Derek F. Holt and Sarah Rees and Claas E. R\"over},
	isbn={978-1-107-15235-9},
	series={London Mathematical Society Student Texts},
	year={2017},
	publisher={Cambridge University Press, Cambridge},
}

@book{hu79,
	author    = {Hopcroft, John E. AND Ullman, Jeffrey D.},
	title     = {Introduction to Automata Theory, Languages, and Computation},
	series    = {Addison-Wesley Series in Computer Science},
	publisher = {Addison-Wesley},
	year      = {1979},
}

@article{kmw17,
	title = {Stallings graphs for quasi-convex subgroups},
	journal = {J. Algebra},
	volume = {488},
	pages = {442--483},
	year = {2017},
	issn = {0021-8693},
	doi = {https://doi.org/10.1016/j.jalgebra.2017.05.037},
	url = {https://www.sciencedirect.com/science/article/pii/S0021869317303642},
	author = {Olga Kharlampovich and Alexei Miasnikov and Pascal Weil}
}

@book{LS79,
	author    = {Lyndon, R. C. AND Schupp, P. E.},
	title     = {Combinatorial Group Theory},
	publisher = {Springer-Verlag, Berlin, Heidelberg, New York},
	year      = {1977},
}

@article{mccammondwise02,
	author = {McCammond, Jonathan P. and Wise, Daniel T.},
	title = {Fans and Ladders in Small Cancellation Theory},
	journal = {Proc.\ London Math.\ Soc.\ },
	volume = {84},
	number = {3},
	pages = {599--644},
	doi = {https://doi.org/10.1112/S0024611502013424},
	url = {https://londmathsoc.onlinelibrary.wiley.com/doi/abs/10.1112/S0024611502013424},
	eprint = {https://londmathsoc.onlinelibrary.wiley.com/doi/pdf/10.1112/S0024611502013424},
	year = {2002}
}

@article{stallings83,
	title={Topology of finite graphs},
	author={John R. Stallings},
	journal={Invent.\ Math.\ },
	year={1983},
	volume={71},
	pages={551--565},
	url={https://api.semanticscholar.org/CorpusID:16643207}
}

@ARTICLE{wise04,
	title     = "Cubulating small cancellation groups",
	author    = "Wise, Daniel T.",
	journal   = "Geom.\ Funct.\ Anal.\ ",
	publisher = "Springer Science and Business Media LLC",
	volume    =  14,
	number    =  1,
	pages     = "150--214",
	month     =  feb,
	year      =  2004
}

\end{document}